\documentclass[12pt]{amsart}
\usepackage{latexsym}
\usepackage{a4}
\usepackage{amssymb}
\usepackage{amsbsy}
\usepackage{bm}
\usepackage{algpseudocode}
\usepackage{enumitem}
\newtheorem{thm}{Theorem}[section] 
 
\newtheorem{lem}[thm]{Lemma} 
 
\newtheorem{cor}[thm]{Corollary} 
\theoremstyle{definition}
 
\newtheorem{prb}[thm]{Problem}
\newtheorem{de}[thm]{Definition}
\numberwithin{equation}{section} 

\newcommand{\N}{\Bbb{ N}} 
\newcommand{\Z}{\Bbb{ Z}} 

\newcommand{\setsuchthat}{\mid}

\newcommand{\ul}[1]{\underline{#1}}
\newcommand{\ol}[1]{\overline{#1}}

\newcommand{\Pot}{\mathcal{P}}

\newcommand{\Ann}{\operatorname{Ann}}

\newcommand{\matrixring}[2]{\operatorname{Mat}_{#1} ( {#2} )}

\newcommand{\MatTwo}[4]
           { \left( \begin{smallmatrix}
                           {#1} & {#2} \\
                           {#3} & {#4}
             \end{smallmatrix} \right) }

\newcommand{\algop}[2]{( {#1}, {#2} )}

\newcommand{\sumd}[2]{\sum_{{\scriptscriptstyle
                             \begin{array}{c}
                               \scriptstyle #1 \\[-1.5mm] \scriptstyle #2
                             \end{array}
                           }}}

\newcommand{\tup}[3]{(#1_{#2},\dots,#1_{#3})}
\DeclareMathAlphabet{\mathbfsl}{OT1}{cmr}{bx}{it}
\DeclareMathAlphabet{\mathsc}{OT1}{cmr}{m}{sc}
\newcommand{\tupBold}[1]{\mathbfsl{#1}}
\newcommand{\vb}[1]{\tupBold{#1}}

\renewcommand{\emptyset}{\varnothing}

\newcommand{\fdeg}{\mathsc{Fdeg}}
\newcommand{\fundeg}{\fdeg}
\newcommand{\partdeg}{\mathsc{Pdeg}}
\newcommand{\pdeg}{\deg_{p}}
\newcommand{\Aug}{\operatorname{Aug}}
\newcommand{\AugZA}{\operatorname{Aug}(\Z[A])}
\newcommand{\range}{\operatorname{Range}}
\title[Chevalley Warning on abelian groups]{Chevalley Warning type
  results on
  abelian groups}

\author{Erhard Aichinger}
\address{
Institut f\"ur Algebra,
Johannes Kepler Universit\"at Linz, Altenberger Strasse 69, 4040 Linz,
Austria}
\email{erhard@algebra.uni-linz.ac.at}

\author{Jakob Moosbauer}
\address{
Institut f\"ur Algebra,
Johannes Kepler Universit\"at Linz, Altenberger Strasse 69, 4040 Linz,
Austria}
\email{jakob.moosbauer@jku.at}
\thanks{Supported by the Austrian Science Fund (FWF):~P29931.}

\urladdr{http://www.jku.at/algebra}

\date{\today}

\begin{document}
\bibliographystyle{amsalpha}
\begin{abstract}
  We develop a notion of \emph{degree} for functions
  between two abelian groups that allows us to
  generalize the Chevalley Warning Theorems from fields
  to noncommutative rings or abelian groups of prime power order.
\end{abstract}

\maketitle
\section{Introduction} \label{sec:intro}
A classical result by C.\ Chevalley \cite{Ch:DDHD} states
that if a system of polynomial
equations $f_1 (x_1, \ldots, x_N) = \cdots =
f_r (x_1, \ldots, x_N)= 0$ over a finite field $F$ has exactly one
solution in $F^N$, then the sum of the total degrees of the
$f_i$'s is at least $N$. E.~Warning \cite{Wa:BZVA} improved this result
by showing that under the hypothesis that $N$ is strictly larger
than the sum of the total degrees of the $f_i$'s,
the number of solutions, which cannot be $1$ by Chevalley's result,
is divisible by the characteristic of the field $F$ (Warning's First Theorem),
and the system has either no or at least $|F|^{N - \sum_{i=1}^r \deg(f_i)}$
solutions (Warning's Second Theorem). Proofs of these results can
be found, e.g., in \cite{As:ANPO}.
All three results have been considerably strengthened:
for Warning's First Theorem, \cite{Ax:ZOPO,Ka:OATO} provide
lower bounds for $\mu$ such that $p^{\mu}$ divides the number of solutions.
O.~Moreno and C.J.\ Moreno showed that in these bounds, the total degree
of a polynomial can be replaced with the \emph{$p$-weight degree} \cite{MM:IOTC}.
S.H. Schanuel and D.J. Katz generalized Chevalley's Theorem to a wider class of
finite commutative rings \cite{Sc:AEOC, Ka:PCDF}. D. Brink 
considered solutions lying in rectangular subsets of $F^N$ (\cite{Br:CTWR},
with a $p$-weight degree version given in \cite{CGM:AITC}).
Brink's result was used in \cite{KS:EOPO, Ai:SSOE} to solve
equations over finite nilpotent rings, groups, and generalizations of these structures.

In this article, we generalize the Chevalley Warning Theorems into a different
direction: instead of polynomial functions on finite fields, we consider
arbitrary functions on abelian groups. Unlike for polynomial functions on fields,
there is no generally agreed concept of the \emph{degree} of such
functions. In the first half of the present article, we develop
a notion of degree for functions between two abelian
groups based on \cite{Ma:TSMP, VL:NIPV}.
Since this degree does not depend on a term representation
of the function, we suggest the name \emph{functional degree} for
this concept.
Using this functional degree, we obtain variants of
Chevalley's and Warning's First Theorem 
for finite abelian $p$-groups (Theorems~\ref{thm:chevgroupA} and~\ref{thm:warngroupA}).
From these results, one easily derives
Moreno and Moreno's $p$-weight improvement 
of Warning's First Theorem (Theorem~\ref{thm:warnA}).
We also see that Warning's First Theorem remains true
if we replace ``finite field'' with
``not necessarily commutative ring of prime power order''
(Theorem~\ref{thm:warnnoncommring}). 
The proof of the last result takes advantage
of the fact that every function, be it polynomial or not,
has a functional degree.
For polynomial functions over finite fields, the functional degree
specializes to the $p$-weight degree from \cite{MM:IOTC}. This allows us
to generalize Asgarli's proof of Warning's Second Theorem \cite{As:ANPO}
to derive its $p$-weight degree improvement 
\cite[Theorem~2]{MM:IOTC} (Theorem~\ref{thm:warningImprovement}).
A similar improvement of Brink's Theorem \cite[Theorem~1]{Br:CTWR}
can be obtained in the special case that the domain is restricted to
a subgroup of $F^N$ (Section~\ref{sec:chevrest}) and we also
obtain that the number of solutions in the subgroup is divisible
by the field characteristic (Theorem~\ref{thm:rest1}).
Warning's First Theorem can be strenghtened if we know
that the functions are not surjective: such  ``restricted range''
versions are given in Theorem~\ref{thm:restrange} and Corollary~\ref{cor:restrange}.

The \emph{functional degree} defined in this note has its
origins in \cite{BA:AFNO,Ma:TSMP,VL:NIPV}. In \cite{Ma:TSMP},
P.\ Mayr defines
the degree of every finitary operation on an algebra with
a Mal'cev term \cite[(3.9)]{Ma:TSMP}. Our definition applies to
functions from one abelian group $A$ into another abelian group $B$,
and it 
involves
the augmentation ideal of the group ring
that acts on a function by shifting its arguments. This follows
an idea from \cite{VL:NIPV}, where such group rings were
successfully applied in the structure theory of nilpotent algebras
in congruence modular varieties (cf. \cite{FM:CTFC}
and
\cite[Corollary~3.10]{Ma:TSMP}).
In those situations where both definitions apply,
Mayr's degree and the functional degree coincide.
A pivotal result is that the functional degree of 
the composed function $g \circ f$ is at most the product of the
functional degrees of $f$ and $g$ (Theorems~\ref{thm:comp} and~\ref{thm:gfs}).
For arbitrary finite abelian groups, there may be functions of
infinite degree, but if 
domain and codomain are finite abelian
$p$-groups (for the same $p$), then the degree of every function
is finite (Section~\ref{sec:findeg}).

\section{Definition of the functional degree}  \label{sec:fundeg}

In this section, we will introduce the \emph{functional degree}
$\fdeg (f)$ of a function $f$ between two abelian groups.
We write $\N$ for the set of positive integers, $\N_0 := \N \cup \{0\}$,
and for $n \in \N$, the set $\{1,2,\ldots, n\}$ is abbreviated
by $\ul{n}$.
In general, we will write groups additively, and we
sometimes simply write $A$ for the abelian group $\algop{A}{+}$.
By $\Z[A]$, we denote its group ring
over the integers \cite{Pa:TASO}.
The elements of this ring are integer tuples
$r = (z_a)_{a \in A}$ indexed by $A$ with $\{a \in A \mid z_a \neq 0\}$
finite. We will write such a tuple in the
form $r = \sum_{a \in A} z_a \tau_a$, where
$z_a \in \Z$ for all $a \in A$,
instead of the more common
$\sum_{a \in A} z_a a$.
The multiplication of $\Z[A]$ then
satisfies  $\tau_a \cdot \tau_b := \tau_{a + b}$ for all $a, b \in A$;
thus $\algop{\Z[A]}{+, \cdot}$ is a commutative ring with
unity $1 = 1 \tau_0$. 
The \emph{augmentation ideal} of $\Z[A]$ is the ideal  
generated by $\{ \tau_{a} - 1 \mid a \in A \}$,
and it will be denoted by $\AugZA$.
For every ideal
$I$ of $\Z[A]$, the power $I^0$ is defined as $\Z[A]$, and $I^n$ is the
ideal generated by $\{i_1 \cdots  i_n \mid i_1, \ldots, i_n \in I\}$.
For $n \in \N$,
$(\AugZA)^n$ is
generated
by the set $\{  \prod_{i = 1}^n (\tau_{a_i} - 1) \setsuchthat a_1, \ldots, a_n \in A \}$.
Let $\algop{B}{+}$ be an abelian group, and let $B^A := \{ f : A \to B \}$.
The ring $\Z[A]$  operates on the group $\algop{B^A}{+}$ by
\[
( \tau_a * f ) \, (x) := f (x + a),
\]
and hence
\[
((\sum_{a \in A} z_a \tau_a) * f) \, (x) = \sum_{a \in A} z_a f (x + a).
\]
We will use this module operation also for a function from an abelian group
into a ring, or even a field, $B$.
The multiplication on $B$ is then
immaterial for the module operation.
For $f: A \to B$ and an ideal $J$ of $\AugZA$,
$J * f := \{ j * f \mid j \in J\}$.
\begin{de}[Functional degree]
  Let $\algop{A}{+}$ and $\algop{B}{+}$ be abelian
  groups, and let $f : A \to B$. Let $I := \AugZA$ be the
  augmentation ideal of $\Z[A]$. The \emph{functional degree}
  of $f$ is defined by
  \[
  \fdeg (f) := \min\, (\{ n \in \N_0 \mid I^{n+1} * f = 0 \}),
  \]
  with $\fdeg (f) = \infty$ if there is no $n \in \N_0$
  with $I^{n+1} * f = 0$.
\end{de}
For all $k, l \in \N$ with $k < l$, we have $I^l \subseteq I^k$, and
thus if $I^k * f = 0$, also $I^l * f = 0$. Hence for each function
$f$ and for each $n \in \N_0$, we have
\[
\fundeg (f) \le n \text{ if and only if } I^{n+1} = 0.
\]
We also see that $\fundeg (f) = \max\, (\{m \in \N \mid I^m * f \neq 0 \} \cup
\{0\})$, where $\max\, (S)$ is $\infty$ for all infinite subsets $S$
of $\N$. Hence for each $n \in \N_0$, we have
\[
\fundeg (f) \ge n \text{ if and only if } (I^n * f \neq 0 \text{ or }
n = 0).
\]
Another description is given in the following Lemma.
\begin{lem}\label{lem:gen} 
  Let $A$ and $B$ be abelian groups such that
  $A$ is generated by $G \subseteq A$.
    Let $f \colon A \to B$, and let
  $m \in \N_0$. Then the following are equivalent:
  \begin{enumerate}
   \item \label{it:gen1} $\fundeg (f) \leq m$.
   \item \label{it:gen2} For all $g_1, \ldots, g_{m+1} \in G$, we have
    \(
      (\prod_{i=1}^{m+1} (\tau_{g_i}-1)) * f = 0.
    \)
  \end{enumerate}
\end{lem}
\begin{proof}
  Let $I:=\AugZA$. We first observe that $I$ is
  generated, as an ideal of $\Z[A]$, by
  $\{ \tau_{g} - 1 \mid g \in G \}$.
  To show this, let $J$ be the ideal generated by
  $\{ \tau_{g} - 1 \mid g \in G \}$. Obviously, $J \subseteq I$.
  For the other inclusion, we note that
  the set $H := \{ h \in A \mid \tau_h - 1 \in J\}$ is a subgroup
  of $A$ because $\tau_{h_1 + h_2} - 1 =
  \tau_{h_1} (\tau_{h_2} - 1) + (\tau_{h_1} - 1)$ and hence,
  if $\tau_{h_1} - 1$ and $\tau_{h_2} - 1$ are both
  elements of $J$, then so is $\tau_{h_1 + h_2} - 1$.
  Furthermore, if $\tau_h - 1 \in J$, then
  $(-\tau_{-h}) \cdot (\tau_h - 1) \in J$, and thus
  $\tau_{-h} - 1 \in J$. Thus $H$ is indeed a subgroup of $A$. 
  Now by the definition of $J$, we have $G \subseteq H$,
  and since $A$ is generated by $G$, $A = H$.
  Thus $\{\tau_a - 1 \mid a \in A \} \subseteq J$,
  which implies $I \subseteq J$.
  We will now prove the equivalence of \eqref{it:gen1} and \eqref{it:gen2}.

  \eqref{it:gen1}$\Rightarrow$\eqref{it:gen2}:
  Clearly, $\prod_{i=1}^{m+1} (\tau_{g_i}-1) \in I^{m+1}$.
  By \eqref{it:gen1}, $I^{m+1} * f = 0$.

  \eqref{it:gen2}$\Rightarrow$\eqref{it:gen1}:
  Since $I$ is generated as an ideal by
  $\{ \tau_g - 1 \mid g \in G \}$, its power $I^{m+1}$
  is generated as an ideal by
  $\{ \prod_{i = 1}^{m+1} (\tau_{g_i} - 1) \mid g_1, \ldots, g_{m+1} \in G \}$.
  By~\eqref{it:gen2}, all these elements lie in the ideal
  $\Ann (f) := \{ r \in \Z[A] \mid r*f = 0\}$ of $\Z[A]$, and
  therefore $I^{m+1} \subseteq \Ann (f)$. Hence
  $I^{m+1} * f = 0$, which implies~\eqref{it:gen1}.
\end{proof}

\section{Elementary properties of the functional degree}

In this section, we list some properties of the functional degree
that follow quite immediately from its definition.
\begin{lem} \label{lem:dprop}
  Let $\algop{A}{+}$ and $\algop{B}{+}$ be abelian groups,
    let $f : A \to B$, let $\sigma \in \AugZA$, and let $a \in A$.
  Then we have:
  \begin{enumerate}
      \item \label{it:dp1} $\fundeg (\tau_a * f) = \fundeg (f)$.
      \item \label{it:dp2} If  $\fundeg (f) > 0$, then
        \[
        \fundeg (f) \ge 1 + \fundeg (\sigma * f).
        \]
      \item \label{it:dp2a} %
        If $\fundeg (f) > 0$, then
        \[
        \fundeg (f) = 1 + \sup \, (\{ \fundeg ( (\tau_b - 1) * f ) \mid b \in A  \}).
        \]
      \item \label{it:dp3} $\fundeg (f) = 0$ if and only if $f$ is a constant function.
      \item \label{it:dphom} If $f$ is a group homomorphism, then
             $\fundeg (f) \le 1$.
  \end{enumerate}
   With the convention $1 + \infty = \infty$, items~\eqref{it:dp2} and~\eqref{it:dp2a}
  are  also valid if $f$ has functional degree $\infty$.
\end{lem}
\begin{proof}
   Let $I := \AugZA$.
  For~\eqref{it:dp1}, we observe that $\tau_a$ is an invertible
  element in the ring $\Z [A]$. The ring $\Z[A]$ is commutative. Therefore,
  for each ideal $J$ of $\Z[A]$, we have
  we have $J * f = 0$ if and only if $J * (\tau_a * f) = 0$.
  Hence $f$ and $\tau_a * f$ have the
same functional degree.
This completes the proof of~\eqref{it:dp1}.

For proving~\eqref{it:dp2}, we first observe that the statment is obvious if
$\fundeg (f) = \infty$. Hence we assume $\fundeg (f) \in \N$ and
set $n := \fundeg (f)$.
  We show that $I^n * (\sigma * f) = 0$. 
To this end, we observe that for all $r \in I^n$,
$r *  (\sigma * f) =  (r \cdot \sigma) *  f$. Since
$r \cdot \sigma \in I^{n+1}$, we have $(r \cdot \sigma)  * f = 0$.
From the definition of the functional degree, we see that
$I^n * (\sigma * f) = 0$ implies that $n - 1 \ge \fundeg (\sigma * f)$,
and therefore $\fundeg (f) \ge 1 + \fundeg (\sigma*f)$.

\eqref{it:dp2a}
From the previous item, we obtain that for every $b \in A$,
\[
\fundeg (f) \ge 1 + \fundeg ((\tau_b - 1) * f).
\]
This implies $\fundeg (f) \ge 1 + \sup\,  (\{ \fundeg ( (\tau_b - 1) * f ) \mid b \in A  \})$.

For proving $\le$, we first show that for every $n \in \N$ with
$n \le \fundeg (f)$ there exists $b \in A$ such that
$n-1 \le \fundeg ( (\tau_b - 1) * f)$. For this purpose,
let  $n \in \N$ be such that $n \le \fundeg (f)$. Then there are
$a_1, \ldots, a_n \in A$ such that
$(\prod_{i=1}^n (\tau_{a_i} - 1)) * f \neq 0$: If this product is $0$
for all $a_1, \ldots, a_n$, then $I^n * f = 0$, contradicting $\fundeg (f) = n$.
Take $b := a_n$. Hence $I^{n-1} * ((\tau_b - 1) * f) \neq 0$, and therefore
$\fundeg ((\tau_b - 1) * f) \ge n - 1$. Hence there exists a $b \in A$ with the
required property.

Therefore, if $\fundeg (f) < \infty$, then taking $n := \fundeg(f)$, we obtain
$\sup \, ( \{ \fundeg ((\tau_b - 1) * f)  \mid b \in A \}) \ge \fundeg (f) - 1$.

In the case $\fundeg (f) = \infty$, we obtain a sequence $(b_n)_{n \in \N}$ from $A$
such that $\fundeg ( (\tau_{b_n} - 1) * f) \ge n - 1$. Hence
$\sup \, (\{ \fundeg ( (\tau_b - 1) * f ) \mid b \in A  \}) = \infty$.

This completes the proof of the $\le$-inequality of~\eqref{it:dp2a}.

\eqref{it:dp3}
For the ``if''-direction, we see that for all $a \in A$,  $(\tau_a - 1) * f = 0$ if $f$ is a constant function. Therefore $I *f = 0$ and thus $\fundeg (f) = 0$. 
For the ``only if''-direction, we assume that $\fundeg (f) = 0$. Now let $x \in A$.
Then $((\tau_{-x} - 1) * f) \, (x) = 0$, hence
$f(x) = f(0)$. Thus $f$ is constant.

\eqref{it:dphom} We assume that $f$ is a group homomorphism, and we let $a, x \in A$. Then
$(\tau_a - 1) * f \, (x)  = f(x+a) - f(x) = f(x) + f (a) - f(x)  = f(a)$, and therefore
$(\tau_a - 1) * f$ is a constant function.
By item~\eqref{it:dp3}, $\fundeg ( (\tau_a - 1) * f ) = 0$. Now by~\eqref{it:dp2a},
$\fundeg (f) \in \{0,1\}$.

\end{proof}

\begin{lem}[Addition]   \label{lem:aprop}
  Let $\algop{A}{+}$ and $\algop{B}{+}$ be abelian groups, and let
  $f, g : A \to B$.
  Then we have:
  \begin{enumerate}
  \item \label{it:d1}
        $\fundeg (f + g) \le \max \, (\fundeg (f), \fundeg (g))$.
  \item \label{it:d2}
    If $\fundeg (f) > \fundeg (g)$, then $\fundeg (f + g) = \fundeg (f)$.
  \end{enumerate}
\end{lem}
\begin{proof}
We set $I := \AugZA$.

\eqref{it:d1} Let $n := \max \, (\fundeg (f), \fundeg (g))$. We assume $n < \infty$ and let
$i \in I^{n+1}$. Then $i * (f + g) = i * f + i * g = 0$, and thus
$\fundeg (f + g) \le n$.

\eqref{it:d2}
Let $n := \fundeg (f)$. If $n < \infty$, then
there is $i \in I^n$ such that $i * f \neq 0$. Since $n > \fundeg (g)$,
we have $i * g = 0$, and thus $i * (f + g) = i * f + i * g \neq 0$.
Therefore $\fundeg (f + g) \ge n$. The converse inequality
follows from item~\eqref{it:d1}.
If $\fundeg (f) = \infty$, then for every $m > \fundeg (g)$, we have
an $i_m \in I^m$ such that $i_m * f \neq 0$. Then
$i_m * (f + g) = i_m * f + i_m * g = i_m * f + 0 \neq 0$, and therefore
$\fundeg (f + g) \ge m$. Hence $\fundeg (f + g) = \infty$. 
\end{proof}
    \begin{lem}[Restriction of a function] \label{lem:rest}
      Let $\algop{A}{+}$ and $\algop{G}{+}$ be abelian groups, let
      $f : A \to G$, let $B$ be a subgroup of $A$, and let
      $f|_B$ be the restriction of $f$ to $B$.
      Then $\fundeg (f|_B) \le \fundeg (f)$. 
    \end{lem}
    \emph{Proof:}
    Suppose that $\fundeg (f) \in \N_0$, and let
    $m := \fundeg (f)$. Let $i \in \Aug (\Z[B])^{m+1}$.
    Seeing $\Z[B]$ as a subring of $\Z[A]$, we observe that
    $i \in \AugZA^{m+1}$. Hence $i * f = 0$.
    For every $r \in \Z[B]$, we have $(r * f)|_B =
    r * (f|_B)$.
    Thus $i * (f|_B) = (i * f)|_B = 0$. 
     Therefore, $\fundeg (f|_B) \le m$. \qed

     \begin{lem}[Combination of functions] \label{lem:combination}
       Let $\algop{A}{+}$, $\algop{B}{+}$ and $\algop{C}{+}$ be
       abelian groups, let $f: A \to B$, and let $g : A \to C$. We define
       a function $h : A \to B \times C$ by $h (a) = (f(a), g(a))$ for all $a \in A$.
       Then $\fundeg (h) = \max \, (\fundeg(f), \fundeg (g))$.
     \end{lem}
     \begin{proof}
       For every $r \in \Z[A]$, we have
       $(r * h) \, (a) = ((r * f) \, (a), (r * g) \, (a))$. Hence for every
       $r \in \Z[A]$, $r * h = 0$ if and only if both $r * f = 0$ and $r * g = 0$,
       which implies the result.
     \end{proof}

\section{The degree of composed functions}
The aim of this section is to prove that the functional degree of a composition
$g \circ f$ is at most the product of the functional degrees of $f$ and $g$.
For this purpose, we characterize the functional degree of a function
by certain ``linearity'' properties.
Similar linearity properties have been used
in \cite{Ma:TSMP} for defining the degree of finitary operations in an
algebra with a Mal'cev term.
For a set $X$ and $n \in \N_0$, we write $\Pot (X)$ for the power set of $X$ and
$\Pot_{\le n} (X)$ for the set
$\{ Y \in \Pot (X) \, : \, |Y| \le n\}$. We write $Y \subset X$ for ($Y \subseteq X$ and $Y \neq X$).
\begin{lem}[Characterization of the degree] \label{lem:char}
  Let $\algop{A}{+}$ and $\algop{B}{+}$ be abelian groups, let $f: A \to B$, and let
  $m \in \N_0$. Then the following are equivalent:
  \begin{enumerate}
  \item \label{it:c1} $\fundeg (f) \le m$.
  \item \label{it:c1a} For every $k > m$, we have
    \begin{equation} \label{eq:bigsum}
      f (\sum_{i=1}^k x_i) = \sum_{S \subset \ul{k}} (-1)^{k - |S| + 1} f (\sum_{j \in S} x_j).
    \end{equation}  
  \item \label{it:c2} For every $k \in \N$, there exists a family $K = \langle \alpha_S \mid S \in \Pot_{\le m} (\ul{k}) \rangle$
    of integers 
    such that for all $x_1, \ldots, x_k \in A$, we have
    \[
    f (\sum_{i=1}^k x_i) = \sum_{S \in \Pot_{\le m} (\ul{k})}  \alpha_S f (\sum_{i \in S} x_i).
    \]
   \item \label{it:c3}
     There exist functions $g_1, \ldots, g_{m+1} : A^{m+1} \to B$ such that for all $x_1, \ldots, x_{m+1} \in A$, we have
     \begin{equation} \label{eq:c3eq}
     f (\sum_{i = 1}^{m+1} x_i) = \sum_{i=1}^{m+1} g_i (x_1, \ldots, x_{m+1}),
     \end{equation}
     and for each $i \in \ul{m+1}$, the function $g_i$ does not depend on its $i$\,th argument.
  \end{enumerate}
\end{lem}
\begin{proof}
  \eqref{it:c1}$\Rightarrow$\eqref{it:c1a}:
  Let $x_1, \ldots, x_k \in A$. Since $k > \fundeg (f)$, we have
  $((\prod_{i = 1}^k (\tau_{x_i} - 1)) * f) \, (0) = 0$.
  We have $\prod_{i=1}^k (\tau_{x_i} - 1) = \sum_{S \subseteq \ul{k}} (-1)^{k - |S|}  \prod_{i \in S} \tau_{x_i}$. Hence
  \[
  \begin{split}
    0 &= ((\prod_{i = 1}^k (\tau_{x_i} - 1)) * f) \, (0)  \\
    &= ((\sum_{S \subseteq \ul{k}} (-1)^{k - |S|}  \prod_{i \in S} \tau_{x_i}) * f) \, (0) \\
    &= \sum_{S \subseteq \ul{k}} (-1)^{k - |S|}  f (\sum_{i \in S} x_i) \\
    &= f (\sum_{i = 1}^k x_i) + \sum_{S \subset \ul{k}} (-1)^{k - |S|}  f (\sum_{i \in S} x_i),
  \end{split}
  \]
  which implies~\eqref{eq:bigsum}.
  
  \eqref{it:c1a}$\Rightarrow$\eqref{it:c2}: We proceed by induction on $k$.
    If $k \le m$, then we set $\alpha_{\ul{k}} = 1$ and $\alpha_S := 0$ for all subsets of $\ul{m}$ that
  are not equal to $\ul{k}$.
  Now assume $k > m$. For all $x_1, \ldots, x_k \in A$, we have
  \begin{equation*}
    f (\sum_{i=1}^k x_i) =\sum_{S \subset \ul{k}} (-1)^{k - |S| + 1} f (\sum_{j \in S} x_j).
  \end{equation*}
  We will now expand each $f (\sum_{j \in S} x_j)$ using the induction
  hypothesis.
  Since every proper subset of $\ul{k}$ has less than $k$ elements,
  the induction hypothesis yields for every $S \subset \ul{k}$ a family
  $K(S) := \langle \alpha^{(S)}_T  \mid T \in \Pot_{\le m} (\ul{k}) \rangle$ of
  integers
  such that
  \[
    f (\sum_{j \in S} x_j) = \sum_{T \in \Pot_{\le m} (\ul{k})}  \alpha^{(S)}_T f (\sum_{i \in T} x_i)
    \]
    for all $\vb{x} \in A^k$.
    Note that we may take $\alpha^{(S)}_T := 0$ for those $T \subseteq \ul{k}$
    with $T \not\subseteq S$.
    Hence
    \[
    \begin{split}
      \sum_{S \subset \ul{k}} (-1)^{k - |S| + 1} f (\sum_{j \in S} x_j) &=
      \sum_{S \subset \ul{k}} (-1)^{k - |S| + 1} \sum_{T \in \Pot_{\le m} (\ul{k})} \alpha^{(S)}_T f (\sum_{i \in T} x_i) \\
      &=
      \sum_{T \in \Pot_{\le m} (\ul{k})}
      (
      \sum_{S \subset \ul{k}} (-1)^{k - |S| + 1}  \alpha^{(S)}_T
      ) \,
      f (\sum_{i \in T} x_i)
    \end{split}
   \]  
    for all $\vb{x} \in A^k$, and therefore
    $\langle \alpha_T \mid T \in \Pot_{\le m} (\ul{k}) \rangle$ with
    $\alpha_T := \sum_{S \subset \ul{k}} (-1)^{k - |S| + 1}  \alpha^{(S)}_T$
   satisfies the required property, which completes the induction step. 

   \eqref{it:c2}$\Rightarrow$\eqref{it:c3}: 
   By~\eqref{it:c2}, we have a family $\langle \alpha_S \mid S \in \Pot_{\le m} (\ul{m+1})\rangle$
   of integers such
   that for all $\vb{x} \in A^k$,
   $f (\sum_{i = 1}^{m+1} x_i) = \sum_{S \in \Pot_{\le m} (\ul{m+1})} \alpha_S f (\sum_{i \in S} x_i)$,
   which is equal to
   \[
   \sum_{j=1}^{m+1} \sumd{
     S \in \Pot_{\le m} (\ul{m+1})}
     {\min (\ul{m+1} \setminus S) = j
   } \alpha_S f (\sum_{i \in S} x_i).
   \]
   Hence
   \[
     g_j (x_1, \ldots, x_{m+1}) := \sumd{
     S \in \Pot_{\le m} (\ul{m+1})}
     {\min (\ul{m+1} \setminus S) = j
     } \alpha_S f (\sum_{i \in S} x_i)
   \]   
   does not depend on its $j$\,th argument, and the $g_j$'s satisfy~\eqref{eq:c3eq}.

   \eqref{it:c3}$\Rightarrow$\eqref{it:c1}:
   We prove that for each $m \in \N_0$ and for each $f : A \to B$, the existence of such $g_1, \ldots, g_{m+1}$ implies $\fundeg (f) \le m$.
   We proceed by induction on $m$.
   For $m = 0$, we observe that the identity $f (x_1) = g_1 (x_1)$ with $g_1$ not depending on $x_1$ implies that
   $f$ is constant, and therefore of functional degree $0$ by
   Lemma~\ref{lem:dprop}\eqref{it:dp3}.
   For the induction step, let $m \in \N$ and assume
   that $f(\sum_{i=1}^{m+1} x_i) = \sum_{i=1}^{m+1} g_i(x_1, \ldots, x_{m+1})$,
      where $g_i$ does not depend on its $i$\,th argument.
   The induction hypothesis is that every $f'$ with
   $f' (\sum_{i=1}^{m} x_i) =  \sum_{i=1}^m h_i(x_1, \ldots, x_{m})$ and
   $h_i$ not depending on its $i$\,th argument satisfies $\fundeg (f') \le
   m-1$.
   We want to prove $\fundeg (f) \le m$.
   If $\fundeg (f) = 0$, there is nothing to prove, hence we may assume that
   $\fundeg (f) \in \N \cup \{\infty\}$.
   We will use Lemma~\ref{lem:dprop}\eqref{it:dp2a} to compute
   $\fundeg (f)$. To this end, we let $b \in A$ and estimate
   $\fundeg ( (\tau_b - 1) * f )$.
   For each $\vb{x} \in A^m$, we have
   \[
   \begin{split}
     ((\tau_b - 1) * f) \, (\sum_{i = 1}^m x_i)  &=
     f (\sum_{i = 1}^m x_i + b) - f (\sum_{i=1}^m x_i + 0) \\
     &=  \sum_{i=1}^{m+1} (g_i (x_1, \ldots, x_m, b) - g_i (x_1, \ldots, x_m, 0)).
   \end{split}
   \]
   Now for $i \in \ul{m}$  and $\vb{x} \in A^m$,
   we define $h_i (x_1, \ldots, x_m) := g_i (x_1, \ldots, x_m, b) - g_i (x_1, \ldots, x_m, 0)$.
   Since $g_{m+1}$ does not depend on its $(m+1)$\,th argument, we have
   \[
          ((\tau_b - 1) * f) \, (\sum_{i = 1}^m x_i)  =
                  \sum_{i=1}^{m} h_i (x_1, \ldots, x_m)
   \]
   for all $\vb{x} \in A^m$. The function $h_i$ does not depend
   on its $i$\,th argument. Therefore,  the induction hypothesis
   yields $\fundeg ( (\tau_b - 1) * f ) \le m-1$.
   Now by Lemma~\ref{lem:dprop}\eqref{it:dp2a}, 
   $\fundeg (f) \le m$.
   \end{proof}
The interplay of the degree with functional composition will be
central in our further development.
   \begin{thm}[Composition]  \label{thm:comp}
     Let $\algop{A}{+}, \algop{B}{+}, \algop{C}{+}$ be abelian groups, let
       $f : A \to B$ and $g : B \to C$ with $\fundeg (f) < \infty$ and $\fundeg(g) < \infty$.
       Then $\fundeg (g \circ f) \le \fundeg (g) \cdot \fundeg (f)$.
   \end{thm}
   \begin{proof}
     Let $m := \fundeg (f)$, $n := \fundeg (g)$. We show that
     $g \circ f$ satisfies condition~\eqref{it:c3} of Lemma~\ref{lem:char}.
     To this end, let $\vb{x} \in A^{mn + 1}$. Then from Lemma~\ref{lem:char}\eqref{it:c2}, applied first to $f$ and then to $g$,  we obtain
     two families
     $\langle \alpha_S \mid  S \in \Pot_{\le m} (\ul{mn+1}) \rangle$ and
     $\langle \beta_T \mid T \in \Pot_{\le n} (\Pot_{\le m} (\ul{mn+1})) \rangle$
     such that
     \[
     \begin{split}
       g ( f ( \sum_{i=1}^{mn+1} x_i)) &=
       g ( \sum_{S \in \Pot_{\le m} (\ul{mn+1})} \alpha_S f ( \sum_{i \in S} x_i)) \\
       &=
       \sum_{T \in \Pot_{\le n}
         (\Pot_{\le m} (\ul{mn+1}))} \beta_T
       g ( \sum_{S \in T} \alpha_S f ( \sum_{i \in S} x_i)).
     \end{split}
     \]
     Now for each $T$, the corresponding summand can only depend
     on those $x_i$ with $i \in \bigcup \{S \mid S \in T\}$, and hence
     on at most $mn$ arguments.
     Therefore we can write $g \circ f \,\, (\sum_{i=1}^{mn+1}x_i)$ as a sum of functions
     each of which depends on at most $mn$ arguments. Collecting these functions
     into $mn + 1$ summands, we obtain the functions $g_1, \ldots, g_{mn+1}$ that satisfy
     condition~\eqref{it:c3} of Lemma~\ref{lem:char}, and hence
     $\fundeg (g \circ f) \le mn$.
   \end{proof}
   Using Theorem~\ref{thm:comp} and items~\eqref{it:c1a} and~\eqref{it:c3}
   of Lemma~\ref{lem:char}, one can prove that
the functional degree coincides with the the degree defined in \cite[(3.9)]{Ma:TSMP} when 
$A = G^l$ and $B = G$ for some abelian group $G$ and some $l \in \N$.

 \section{Partial degree}

   We will now define the \emph{partial} degree of a function $f : \prod_{j=1}^k A_j \to C$ in each of
   its variables.
   Intuitively, $\partdeg_i (f)$ is the maximal degree of those
   functions from $A_i \to C$ that we obtain by setting all arguments
   except for the $i$\,th one to constants. More formally,
   we proceed as follows:
   For $\vb{a} = \tup{a}{1}{k} \in  \prod_{j=1}^k  A_j$ and $i \in \ul{k}$, we define
   the function $E^{(i)}_{\vb{a}} : A_i \to \prod_{j=1}^k A_j$ by
   $E^{(i)}_{\vb{a}} (x) \, (j) := a_j$ if $j \neq i$, and
   $E^{(i)}_{\vb{a}} (x) \, (i) := x$ for all $x \in A_i$.
    Hence 
   $E^{(i)}_{\vb{a}} (x) = (a_1, \ldots, a_{i-1}, x, a_{i+1}, \ldots, a_k)$.
    \begin{de} [Partial degree]
       Let $k \in \N$, let $\langle A_j \mid j \in \ul{k} \rangle$ be a family of abelian groups, and let $C$ be
       an abelian group. Let $f : \prod_{j=1}^k A_j \to C$, and let $i \in \ul{k}$.
       Then the \emph{partial degree of $f$ in its $i$\,th argument},  $\partdeg_i (f)$, is defined
     by
     \[
       \partdeg_i (f) := \sup \, (\{ \fundeg (f \circ E^{(i)}_{\vb{a}}) \mid
                                  \vb{a} \in \prod_{j = 1}^k A_j \}).
     \]
   \end{de}

    \begin{thm} \label{thm:partialtotal}
       Let $k \in \N$, let $\langle A_j \mid j \in \ul{k} \rangle$ be a family of abelian groups,
      let $C$ be an abelian group, and let $r \in \ul{k}$.
      Let $f : \prod_{j=1}^k A_j \to C$.
    Then $\partdeg_r (f) \le \fundeg (f) \le \sum_{j=1}^k \partdeg_j (f)$. 
   \end{thm}
    \begin{proof}
      For proving the first inequality $\partdeg_r (f) \le \fundeg (f)$,
      we fix $\vb{a} \in \prod_{j=1}^k A_j$ and estimate the
      degree of the function $f \circ E^{(r)}_{\vb{a}}$ from $A_r$ to $C$.
      Since the degree of a constant function is $0$ and the degree
      of the identity mapping on $A_r$ is at most $1$ (it is $0$ if
      $|A_r|  = 1$ and $1$ if $|A_r| > 1$), Lemma~\ref{lem:combination}
      yields that the function $E^{(r)}_{\vb{a}}$ is of degree at most $1$.
      Hence by Theorem~\ref{thm:comp},
      $\fundeg (f \circ E^{(r)}_{\vb{a}}) \le \fundeg (f) \cdot 1$,
      which completes the proof of $\partdeg_r (f) \le \fundeg (f)$.
    
      We will now prove 
      \begin{equation} \label{eq:f2}
        \fundeg (f) \le \sum_{j=1}^k \partdeg_j (f).
      \end{equation}  
      We first consider the case $k = 2$.
      To this end, let $A := A_1$, $B := A_2$, and $f : A \times B \to C$.
      Let $m := \partdeg_1 (f)$ and $n := \partdeg_2 (f)$.
      We assume $m < \infty$ and $n < \infty$. 
  Let $I := \Aug (\Z[A])$, $J := \Aug (\Z[B])$, $K := \Aug (\Z [A \times B])$.
  Let $\varphi_A : \Z[A] \to \Z[A \times B]$ be the ring homomorphism defined
  by $\varphi_A (\tau_a) := \tau_{(a,0)}$ for all $a \in A$, and let
  $\varphi_B : \Z[B] \to \Z [A \times B]$ be defined by $\varphi_B (\tau_b) := \tau_{(0,b)}$.
  We observe that
  \[
      \begin{array}{rcl}
        (\tau_{a} * (f \circ E^{(1)}_{(a',  b)}))  \, (x) & = & (\varphi_A (\tau_{a}) * f) \, (x, b) \text{ and } \\
        (\tau_{b} * (f \circ E^{(2)}_{(a,  b')}))  \, (y) & = & (\varphi_B (\tau_{b}) * f) \, (a, y)
      \end{array}
      \]
    for all $a, a', x \in A$ and $b, b', y \in B$.
    Therefore   for all elements $r = \sum_{a \in A} z_a \tau_a \in \Z[A]$
    and $s = \sum_{b \in B} z'_b \tau_b \in \Z[B]$,
    we have
    \begin{equation} \label{eq:phiNew}
      \begin{array}{rcl}
       (\varphi_A (r) * f) \, (x, b)  & = &  (r * (f \circ E^{(1)}_{(a, b)}))  \, (x)   \text{ and}  \\
       (\varphi_B (s) * f) \, (a, y)  & = &  (s * (f \circ E^{(2)}_{(a, b)}))  \, (y)
      \end{array}
    \end{equation}
      for all $x,a \in A$ and $y, b \in B$.

  Since $\varphi_A (I) \subseteq K$, we have $\varphi_A (I^l) \subseteq K^l$, and similarly
  $\varphi_B (J^l) \subseteq K^l$ for all $l \in \N_0$.
  We let $\hat{I} := \langle \varphi_A (I) \rangle$ be the ideal of $\Z[A \times B]$ that
  is generated by $\varphi_A (I)$, and similarly $\hat{J} := \langle \varphi_B (J) \rangle$
  denotes the ideal generated by $\varphi_B (J)$.
  We then have
  \begin{equation} \label{eq:K1}
    K \subseteq \hat{I} + \hat{J}.
  \end{equation}  
  To see this, we notice that $\tau_{(a,b)} - 1 =
  \tau_{(0,b)} \cdot (\tau_{(a,0)} - 1) + (\tau_{(0,b)} - 1)  =
  \tau_{(0,b)} \cdot \varphi_A (\tau_a - 1) + \varphi_B (\tau_b -1)$, completing the
  proof of~\eqref{eq:K1}.
  Now we are ready to prove $\fundeg (f) \le \partdeg_1 (f) + \partdeg_2 (f)$.
  Let $m := \partdeg_1 (f)$ and $n := \partdeg_2 (f)$, and assume that $m,n \in \N_0$.
  We prove $K^{m+n+1} * f = 0$.
  We know $K^{m+n+1} \subseteq (\hat{I} + \hat{J})^{m+n+1}
  \subseteq \sum_{i = 0}^{m+n+1} \hat{I}^i \hat{J}^{m+n+1-i}
  \subseteq \hat{I}^{m+1} + \hat{J}^{n+1}$.
  We first prove
  \begin{equation} \label{eq:I} 
    \hat{I}^{m+1} * f  = 0.
  \end{equation}
  The ideal $\hat{I}$ is generated as an ideal by
  $\varphi_A (I)$. Therefore, $\hat{I}^{m + 1}$ is generated
  (as an ideal of $\Z[A \times B]$) by
  $P := \{ \varphi_A (i) \mid i \in I^{m+1} \}$.
  Next, we show that  for each $\sigma \in I^{m+1}$, we have
  \begin{equation} \label{eq:sf0}
    \varphi_A (\sigma) * f  = 0.
  \end{equation}
  For this purpose, we observe that
  for all $x,y \in A$, we have
  $(\varphi_A (\sigma) * f) \,  (x, y) = (\sigma * (f \circ E^{(1)}_{(a,y)})) \,(x) $. Since
  $\sigma \in I^{m+1}$ and $\fundeg  (f \circ E^{(1)}_{(a,y)})  \le \partdeg_1 (f) = m$,
  we have $(\sigma * (f \circ E^{(1)}_{(a,y)})) \,( x ) = 0$, completing the proof of~\eqref{eq:sf0}.
   Since $L := \{\psi \in \Z[A \times B] \mid \psi * f \} = 0$ is an ideal
   of $\Z[A \times B]$ and since by~\eqref{eq:sf0}, $P \subseteq L$,
   we have  $\hat{I}^{m+1} \subseteq L$,
  which implies~\eqref{eq:I}.
  Similarly, $\hat{J}^{n+1} * f = 0$.
  Hence $\hat{I}^{m+1} \subseteq L$ and $\hat{J}^{n+1} \subseteq L$, and therefore
  $\hat{I}^{m+1} + \hat{J}^{n+1} \subseteq L$.
  Thus $K^{m+n+1} \subseteq L$, which implies $K^{m+n+1} * f = 0$.
  This proves that the functional degree of $f$ is at most $m + n$, which completes the proof of~\eqref{eq:f2} for the  case $k = 2$.

  For an arbitrary $k \in \N$, we proceed by induction on $k$. In the case
  $k = 1$, the assertion is obvious, and the case $k = 2$ has been treated above.
  Let us now assume $k > 2$. Let $A := \prod_{i = 1}^{k-1} A_i$ and $B := A_k$, and
  let $g : A \times B \to C$ be defined by
  \[
    g ( (x_1, \ldots, x_{k-1}),   x_k)) := f (x_1,\ldots, x_k)
    \]
    for all $(x_1, \ldots, x_{k-1}) \in A$ and $x_k \in B$.
    This allows us to view $f$ as a function in $k$ and $g$
    as a function in $2$ arguments.
    Let $\varphi$ be the isomorphism from $A \times B$
    to $\prod_{i=1}^k A_i$ defined by $\varphi ((x_1, \ldots, x_{k-1}), x_k)
    = (x_1, \ldots, x_k)$.
    Applying Theorem~\ref{thm:comp} and Lemma~\ref{lem:dprop}\eqref{it:dphom}
    to $f = g \circ  \varphi^{-1}$ and $g = f \circ \varphi$, we see
     $\fundeg (f) = \fundeg (g)$.
    The domain of $g$ is the product of the $2$ groups
  $A$ and $B$, and so we can apply the case above and obtain
  \[
  \fundeg (g) \le \partdeg_1 (g) + \partdeg_2 (g).
  \]
  We have
  \[
     \partdeg_1 (g) =
     \sup \, (\{ \fundeg (g \circ E^{(1)}_{(\vb{a}, b)}) \mid (\vb{a}, b) \in A \times B \}).
  \]
  Let $(\vb{a}, b) \in A \times B$. The function $g \circ E^{(1)}_{(\vb{a}, b)}$ is a function from
  $A$ to $C$, and it satisfies
  \[
  g \circ E^{(1)}_{(\vb{a}, b)} ( (x_1, \ldots, x_{k-1}) )
  =
  g ( (x_1, \ldots, x_{k-1}), b)
  =
  f (x_1, \ldots, x_{k-1}, b)
  \]
  for all $(x_1, \ldots, x_{k-1}) \in A$.
  Let $h : A \to C$ be defined by
  \[
  h  (x_1, \ldots, x_{k-1}) := f (x_1, \ldots, x_{k-1}, b).
  \]
  By the induction hypothesis, we have
  $\fundeg (h) \le \sum_{i=1}^{k-1} \partdeg_i (h)$.
  For $i \in \ul{k-1}$, we have
  \[
  \{ h \circ E^{(i)}_{\vb{a}'} \mid \vb{a}' \in A \} =
  \{ f \circ E^{(i)}_{\vb{a}'} \mid \vb{a}' \in \prod_{j=1}^k A_j, a'_k = b \} \subseteq
  \{ f \circ E^{(i)}_{\vb{a}'} \mid \vb{a}' \in \prod_{j=1}^k A_j \}.
  \]
  Taking the suprema of the degrees of the functions in these sets, we obtain
  $\partdeg_i (h) \le \partdeg_i (f)$. Altogether, we obtain
  $\partdeg_1 (g) \le \sum_{i=1}^{k-1} \partdeg_i (f)$.
  Since $g \circ E^{(2)}_{(\vb{a}, b)} = f \circ E^{(k)}_{ (a_1, \ldots, a_{k-1}, b)}$ for
  all $\vb{a} = (a_1, \ldots, a_{k-1}) \in A$ and $b \in A_{k}$, we have
  $\partdeg_2 (g) = \partdeg_k (f)$. 
  Hence $ \fundeg (f) \le \sum_{i=1}^k \partdeg_i (f)$. 
   \end{proof}     
    
    \begin{thm} \label{thm:gfs}
     Let $k \in \N$, $\langle B_i \mid i \in \ul{k} \rangle$ be a family of abelian groups,
     let $A$ and $C$ be abelian groups, and for each $i \in \ul{k}$, let
     $f_i : A \to B_i$.
     Let $g : \prod_{i=1}^k B_i \to C$, and let $G : A \to C$ be defined
     by
     \[
           G (a) := g (f_1  (a) , \ldots, f_k (a))
     \]
     for all $a \in A$.
     Then $\fundeg (G) \le \sum_{i=1}^k \partdeg_i (g) \cdot \fundeg (f_i)$.
   \end{thm}
   \begin{proof}
     We first define a function
     $h : A^k  \to C$ by
     \[
     h (x_1, \ldots, x_k) := g (f_1 (x_1), \ldots, f_k (x_n))
     \]
     for all $\vb{x} \in A^k$.
     We will first estimate $\fundeg (h)$. By Theorem~\ref{thm:partialtotal},
     we have
     $\fundeg (h) \le \sum_{i = 1}^k \partdeg_i (h)$. Now let $i \in \ul{k}$.
     For giving an upper bound for $\partdeg_i (h)$, we fix $\vb{a} \in A^k$
     and compute $\fundeg (h \circ E^{(i)}_{\vb{a}})$.
     We have
     $h \circ E^{(i)}_{\vb{a}} = (g \circ E^{(i)}_{(f_1 (a_1), \ldots, f_k (a_k))}) \circ f_i$.
     The functional degree of $g \circ E^{(i)}_{(f_1 (a_1), \ldots, f_k (a_k))}$ is at most
     $\partdeg_i (g)$. Hence by Theorem~\ref{thm:comp},
     $\fundeg (h \circ E^{(i)}_{\vb{a}})  = \fundeg ((g \circ E^{(i)}_{(f_1 (a_1), \ldots, f_k (a_k))}) \circ f_i)
     \le \partdeg_i (g) \cdot \fundeg (f_i)$.
     Taking the supremum we obtain $\partdeg_i (h) \le \partdeg_i (g) \cdot \fundeg (f_i)$. Thus, we have
     \[
     \fundeg (h) \le \sum_{i=1}^k \partdeg_i (g) \cdot \fundeg (f_i).
     \]
     Let us now find the degree of $G$. To this end,
     let $e : A \to A^k$, $e(a) = (a,\ldots, a)$. Then
     by Lemma~\ref{lem:combination}, $\fundeg (e) \le 1$, and therefore
     $\fundeg (G) = \fundeg (h \circ e) \le \fundeg (h) \le \sum_{i=1}^k \partdeg_i (g) \cdot \fundeg (f_i)$. 
   \end{proof}     
     
   \section{Multiplicative properties of the functional degree}
   
\begin{lem}[Multiplicative property of the functional degree]   \label{lem:mprop}
  Let $\algop{A}{+}$ be an abelian group, let
  $\algop{R}{+,\cdot}$ be a (not necessarily commutative)
  ring, and
  let $f, g : A \to R$.
  Then $\fundeg (f \cdot g) \le \fundeg (f) + \fundeg (g)$.
\end{lem}
\begin{proof}
Let $m : R \times R \to R$, $(x,y) \mapsto x \cdot y$.
Since for all $a, b \in R$, both mappings $y \mapsto a \cdot y$ and
$x \mapsto x \cdot b$ are group homomorphisms, Lemma~\ref{lem:dprop}\eqref{it:dphom}
yields $\partdeg_1 (m) \le 1$ and $\partdeg_2 (m) \le 1$.
Now Theorem~\ref{thm:gfs} implies
$\fundeg (f \cdot g) = \fundeg (m (f, g)) \le 1 \cdot \fundeg (f) + 1 \cdot \fundeg (g) =
\fundeg (f) + \fundeg (g)$. 
\end{proof}

A lower bound to the functional degree of certain functions is
provided by the next lemma.
  For two groups $\algop{A}{+}$ and $\algop{B}{+}$, a field $(F,+,\cdot)$,
  and $f : A \to F$ and $g : B \to F$, we let $f \otimes g : A \times B \to F$
  be defined by $(f \otimes g) \, (a, b) := f (a) \cdot g (b)$.
  \begin{lem} \label{lem:prodprop}
    Let $\algop{A}{+}$ and $\algop{B}{+}$
    be abelian groups, let
    $\algop{F}{+,\cdot}$ be a field, and let 
     $f : A \to F$ and $g : B \to F$. 
    We assume that
    none of $f$ and $g$ is the zero function.
    Then $\fundeg (f \otimes g) = \fundeg (f) + \fundeg (g)$.
  \end{lem}  
  \begin{proof}
  Let $I := \Aug (\Z[A])$, $J := \Aug (\Z[B])$, $K := \Aug (\Z [A \times B])$.
  Let $\varphi_A : \Z[A] \to \Z[A \times B]$ be the ring homomorphism defined
  by $\varphi_A (\tau_a) := \tau_{(a,0)}$ for all $a \in A$, and let
  $\varphi_B : \Z[B] \to \Z [A \times B]$ be defined by $\varphi_B (\tau_b) := \tau_{(0,b)}$.
  We see that $\varphi_A (\tau_a) * (f \otimes g) = (\tau_a * f) \otimes g$ and
  $\varphi_B (\tau_b) * (f \otimes g) = f \otimes (\tau_b * g)$.
  Next, we show that for all elements $r = \sum_{a \in A} z_a \tau_a \in \Z[A]$
  and $s = \sum_{b \in B} z'_b \tau_b \in \Z[B]$, 
  we have
  \begin{equation} \label{eq:phi}
    (\varphi_A (r) \cdot \varphi_B (s)) * (f \otimes g) = (r * f) \otimes (s * g).
  \end{equation}
  We have
  \[
  \begin{split}
   (r * f) \otimes (s * g) \, (x,y) &= (\sum_{a \in A} z_a f (x+a))(\sum_{b \in B} z'_b g (y + b))
   \\ &  = \sum_{a \in A} \sum_{b \in B} z_a z'_b f(x + a) g (y + b) 
   \\ & = \sum_{a \in A} \sum_{b \in B} z_a z'_b \, \big(
       (\tau_{(a,b)} * (f \otimes g)) \, (x, y) \big) 
   \\ & = \sum_{a \in A} \sum_{b \in B} z_a z'_b \, \big(
   ((\tau_{(a,0)} \cdot \tau_{(0,b)}) * (f \otimes g)) \, (x, y) \big)   \\ & =
   \big( (\sum_{a \in A} \sum_{b \in B} z_a z'_b \tau_{(a,0)} \cdot \tau_{(0,b)}) * (f \otimes g) \big) \,\, (x, y)
   \\ & =
    \Big( \big((\sum_{a \in A} z_a \varphi_A (\tau_a)) \cdot (\sum_{b \in B} z'_b \varphi_B (\tau_b)) \big) * (f \otimes g) \Big) \, (x,y)
   \\ & =
   ((\varphi_A (r) \cdot \varphi_B (s)) * (f \otimes g)) \, (x , y),
  \end{split}
  \]
  which proves~\eqref{eq:phi}.
  Since $\varphi_A (I) \subseteq K$, we have $\varphi_A (I^k) \subseteq K^k$, and similarly
  $\varphi_B (J^k) \subseteq K^k$ for all $k \in \N_0$.

  Let $m,n \in \N_0$ be such that
  $m \le \fundeg (f)$ and $n \le \fundeg (g)$. Since $f \neq 0$ and $g \neq 0$, 
  we have $I^m * f \neq 0$ and $J^n * g \neq 0$, and thus
  there are $i \in I^m$ and
  $j \in J^n$ such that
  $i * f \neq 0$ and $j * g \neq 0$.
  We have $(\varphi_A (i) \cdot \varphi_B (j)) * (f \otimes g) =
  (i * f) \otimes (j * g)$. Now if $a \in A$ and $b \in B$ are
  such that $(i*f) \, (a) \neq 0$ and $(j * g) \, (b) \neq 0$,
  then $((i * f) \otimes (j * g))\, (a,b) \neq 0$, and therefore
  $(i * f) \otimes (j * g) \neq 0$.
  Since $\varphi_A (i) \cdot \varphi_B (j) \in K^{m + n}$, we obtain
  $K^{m + n} * (f \otimes g) \neq 0$, which implies
  $\fundeg (f \otimes g) \ge m + n$. Hence we have $\fundeg (f \otimes g) \ge \fundeg (f) + \fundeg (g)$.

  For the other inequality,
  we let $p_1 : A \times B \to A$, $(a,b) \mapsto a$ and
  $p_2 : A \times B \to B$, $(a,b) \mapsto b$.
  By Lemma~\ref{lem:dprop}\eqref{it:dphom},
  $\fundeg (p_1) \le 1$ and
  $\fundeg (p_2) \le 1$.
    Let $\hat{f} : A \times B \to R$, $(a, b) \mapsto f(a)$, and
  $\hat{g} : A \times B \to R$, $(a, b) \mapsto g(b)$. Then
  $f \otimes g = \hat{f} \cdot \hat{g}$.
  Now $\fundeg (f \otimes g) = \fundeg (\hat{f} \cdot \hat{g})
               \le \fundeg (\hat{f}) + \fundeg (\hat{g}) =
               \fundeg (f \circ p_1) + \fundeg (g \circ p_2) \le
               \fundeg (f) \cdot 1 + \fundeg (g) \cdot 1$.
  \end{proof}

\section{Functions of maximal degree} \label{sec:genbounds}
  We will need upper bounds for the degrees
  of functions between two finite abelian groups,
  and we will also need examples of functions for which
  these bounds are attained.
  For two abelian groups $A, B$, we define
  \[
  \delta (A, B) := \sup \, (\{ \fundeg (f) \mid f \in B^A \}).
  \]
  For $a \in A$ and $b \in B$, we define the
  characteristic function $\chi_{a}^{A,b}$ of $a$ with value $b$ by
  \[
     \chi_a^{A,b} (a) = b \text{ and }
     \chi_a^{A,b} (x) = 0 \text{ for all }x \in A \setminus\{a\}.
  \]    
  If $A$ and $B$ have nonzero elements of coprime order,
  then $\delta(A,B) = \infty$. To this end, we first show:
  \begin{lem} \label{lem:pq}
    Let $p,q$ be different primes, let $P$ be
    an abelian $p$-group, and let $Q$ be an abelian
    $q$-group.
    Then every function $f:P \to Q$ of finite functional degree
    is constant.
  \end{lem}
  \begin{proof}
    We will show that there is no function of
    degree $n$ with $n \in \N$.
    We proceed by induction. If the degree
    of $f$ is~$1$, then
    for all $a,b \in P$, we have
    $0 = (((\tau_b - 1) (\tau_{a} -1)) * f) \, (0) =
      f(a+b) - f(a) + f(0) - f (b)$, and therefore
      $g$ defined by $g(x)= f(x) - f(0)$ is a group
      homomorphism from $P$ to $Q$.
      Now for every $x \in P$, the order of $x$ is $p^m$ for some
      $m \in \N_0$. Hence if $g(x) \neq 0$,
      the order of $g(x)$ is $p^{m_1} = q^{m_2}$ for some
      $m_1, m_2 \in \N_0$. Thus $m_1= m_2 = 0$, 
          which implies $g(x) = 0$. 
      Therefore, $f$ is a constant mapping, and thus its
      degree is $0$. Hence there is no $f$ of degree $1$.
      For the induction step, let $n \in \N$. Suppose that
      there is a function $f$ of degree $n + 1$.
      Then by Lemma~\ref{lem:dprop}\eqref{it:dp2a},
      there also exists a function $f'$ of degree $n$. But no such
      function exists by the induction hypothesis. Hence
      there is no function $f$ of degree $n+1$ either.
   \end{proof}   
  \begin{thm} \label{thm:deltafinite}
    Let $A$ and $B$ be periodic abelian groups.
    If $\delta (A, B)$ is finite, then there is
    a prime $p$ such that both $A$ and $B$ are $p$-groups.
  \end{thm}
  \begin{proof}
    Suppose that there are primes $p, q$ with
    $p \neq q$ such that $A$ has an element $a$ of order
    $p$ and $B$ has an element $b$ of order $q$. We claim
    that $\fundeg (\chi_0^{A, b}) = \infty$.
      Let $\langle a \rangle$ be the subgroup generated by $a$.
      Lemma~\ref{lem:rest} yields
      $\fundeg (\chi_0^{A, b}) \ge \fundeg (\chi_0^{A, b}|_{\langle a \rangle})$.
      This restriction is a nonconstant mapping from the $p$-group
      $\langle a \rangle$ into
      the $q$-group $\langle b \rangle$, which has infinite degree
      by Lemma~\ref{lem:pq}.
  \end{proof}
  Before we investigate $\delta (A, B)$ for finite abelian
  $p$-groups, we relate $\delta (A, B)$ to the nilpotency
  degree of the augmentation ideal in the group ring
  (cf. \cite[Corollary~3.10]{Ma:TSMP}).
  Let $R$ be a commutative ring with $1$, and let $R[A]$
  denote the group ring of $A$. Then we define
  \[
  \nu (R [A]) := \min \, \{ m \in \N \mid (\Aug (R[A]))^m = 0 \},
  \]
  with $\nu (R [A]) = \infty$ if no such $m$ exists.

  For elements $r \in \Z[A]$, the ideal generated by
  $r$ will be denoted by $(r)$; if $n \in \Z$, then
  $(n) = (n \tau_0)$.
  \begin{lem} \label{lem:deltanu}
    Let $A$ be an abelian group, and let $B$ be
    an abelian group of finite exponent~$n \ge 2$. Then
    $\delta (A, B) + 1 = \nu (\Z_n [A])$ (with the convention
    $\infty + 1 = \infty$).
  \end{lem}
  \begin{proof}
     We first prove $\delta (A, B) + 1 \le \nu (\Z_n [A])$.
     If $\nu (\Z_n [A]) = \infty$,  there is nothing to prove,
     so we assume $\nu (\Z_n [A]) = m \in  \N$. We then have
    $(\Aug (\Z_n [A]))^m = 0$. From this, we conclude
    that $(\AugZA)^m \subseteq (n)$ in $\Z[A]$.
    Hence $(\AugZA)^m * f = 0$ for every $f \in B^A$, and therefore
    $\fundeg (f) \le m-1$. This implies $\delta (A, B) + 1 \le m$.

    For the inequality $\delta (A, B) + 1 \ge \nu (\Z_n [A])$,
    we assume $\delta (A, B) < \infty$. Then
    $\AugZA^{\delta (A, B) + 1} * f = 0$ for every $f \in B^A$.
    We have $\Z_n [A] = \Z [A] / (n)$, and it acts faithfully
    on $B^A$ by $(r + (n))*  f := r*f$. To prove faithfulness,
    we fix an element
    $b \in B$ of order $n$.
    Now suppose that $(r + (n)) * \chi_0^{A, b} = 0$.
    Let $r = \sum_{a \in A} z_a \tau_a$.
    Then $0 = r * \chi_0^{A,b} =  
        \sum_{a \in A} z_a (\tau_a * \chi_{0}^{A, b})
          = \sum_{a \in A} z_a \chi_{-a}^{A, b}$.
            The value of this function at $x \in A$ is
            $z_{-x} b$. From $0 = z_{-x} b$, we obtain
            that $n$ divides $z_{-x}$. Altogether, $r \in (n)$,
            completing the proof that $\Z_n [A]$ acts faithfully
            on $B^A$.
            Hence $\AugZA^{\delta (A, B) + 1} \subseteq (n)$,
            and therefore $\Aug (\Z_n [A])^{\delta (A, B) + 1} = 0$.
            Thus $\nu (\Z_n [A]) \le \delta (A, B) + 1$. \qed

            Often, certain characteristic functions attain
            the maximal possible degree.
               \begin{lem} \label{lem:deltachi}
       Let $A$ be an abelian group, let $B$ be
       an abelian group of finite exponent~$n$, and let
       $b \in B$ be an element of order $n$.
       Then
       $\delta (A,B) = \fundeg (\chi_{0}^{A,b})$.
   \end{lem}
   Since $\delta (A, B) \ge \fundeg (\chi_{0}^{A,b})$ follows
   from the definition of $\delta (A, B)$, we only have
   to prove $\delta (A, B) \le \fundeg (\chi_0^{A, b})$.
   We first show that for every $a \in A$ and $b' \in B$, we have
   \begin{equation} \label{eq:chis}
   \fundeg (\chi_0^{A, b}) \ge  \fundeg (\chi_a^{A, b'}).
   \end{equation}
   We first construct
   a group endomorphism $h$ of $B$
   such that $h (b) = b'$. Such an endomorphism can be constructed
   by constructing an endomorphism $h_p$ in each $p$-component $B_p$
   of $B$ \cite[4.1.1]{Ro:ACIT}. We will now show that
   in an abelian $p$-group $B_p$ of
   finite exponent, an element $b_p$ of
   maximal order $p^{\beta}$ can be mapped to each element $b_p' \in B_p$.
   By \cite[4.2.7]{Ro:ACIT}, $b_p$ generates a direct factor of $B_p$,
   thus
   there is an epimorphism $e : B_p \to \Z_{p^{\beta}}$ with
   $e (b_p) = 1$. We define
   $i : \Z_{p^{\beta}} \to  B_p$ by $i (z) = z b_p'$,
   and $h_p := i \circ e$.
   From these $h_p$'s, we obtain a group endomorphism of $B$ with
   $h(b) = b'$.
    Since
   $\chi_a^{A, b'} = h \circ (\tau_{-a} * \chi_0^{A, b})$,
    Theorem~\ref{thm:comp} and Lemma~\ref{lem:dprop} yield
   $\fundeg (\chi_a^{A, b'} ) \le 1 \cdot \fundeg (\chi_0^{A, b})$,
   which completes the proof of~\eqref{eq:chis}.
   Now let $f \in B^A$ such that $\fundeg (f) \ge m \in \N_0$.
   We will show that then $\fundeg (\chi_0^{A, b}) \ge m$.
   Since $\fundeg (f) \ge m$, there are
   $a_1, \ldots, a_m \in A$ and there is $y \in A$ such that
   $((\prod_{i = 1}^m (\tau_{a_i} - 1)) * f) \, (y)  \neq 0$.
   Let $A'$ be a finite subset of $A$ and $(z_a)_{a \in A'} \in \Z^{A'}$
   be such that
   $\prod_{i = 1}^m (\tau_{a_i} - 1)
    = \sum_{a \in A'} z_a \tau_a$.
    Let $g := \sum_{a \in A'} \chi_{y + a}^{A, f(y + a)}$.
    Then for $S = \{ y + a \mid a \in A' \}$, we have
    $g|_S = f|_S$.
    Hence $0 \neq ((\prod_{i = 1}^m (\tau_{a_i} - 1)) * f) \, (y)  =
           ((\sum_{a \in A'} z_a \tau_a) * f) \, (y) =
            \sum_{a \in A'} z_a f (y + a) =
            \sum_{a \in A'} z_a g (y + a) =
            ((\sum_{a \in A'} z_a \tau_a) * g) \, (y) =
            ((\prod_{i = 1}^m (\tau_{a_i} - 1)) * g) \,  (y)$.
            Thus $\fundeg (g) \ge m$. Hence there is
            an $a \in A'$ such that
            $\fundeg (\chi_{y + a}^{A, f(y + a)}) \ge m$,
      and thus by \eqref{eq:chis}, $\fundeg (\chi_{0}^{A, b}) \ge m$.
      This implies $\fundeg (\chi_0^{A, b}) \ge \delta (A, B)$.
 \end{proof}      

  \section{Bounds for the functional degree in finite $p$-groups} \label{sec:bounds}
  We will now determine upper bounds for
  $\delta (A, B)$, where
  $A = \prod_{i=1}^k \Z_{p^{\alpha_i}}$ and $B$ is an abelian group of exponent
  $p^{\beta}$.
  In \cite[Corollary~2.5]{Ka:TJRO}, we find that
  $\nu (\Z_p [A]) = 1 + \sum_{i=1}^k (p^{\alpha_i} - 1)$, and
  therefore Lemma~\ref{lem:deltanu} yields 
  $\delta (A, \Z_p) = \sum_{i=1}^k (p^{\alpha_i} - 1)$.
  For computing an upper bound for $\delta (A, B)$, we
  let $m := \delta (A, \Z_p)$ and observe
  that the augmentation ideal $J$ of $\Z_p [A]$
  satisfies $J^{m+1} = 0$. Therefore
  in $\Z[A]$, $(\AugZA)^{m + 1} \subseteq (p)$, and thus
  $(\AugZA)^{\beta (m+1)} \subseteq (p^{\beta})$, which implies
  $(\AugZA)^{\beta (m+1)} * B^A\ = 0$. Thus every function from $A$ to $B$
  is of functional degree at most $\beta (m+1) - 1$.

  We include a self-contained derivation of these results
  without resorting to the results on the nilpotency degree of the augmentation
  ideal from \cite{Ka:TJRO}, and without the results
  from Section~\ref{sec:genbounds}. In the
  case $\beta \ge 2$, we obtain a bound which is lower than the one obtained
  through the ring theoretic considerations given above.
      \begin{lem} \label{lem:pgroupupperbound}
    Let $p$ be a prime, let $k \in \N$,
    let $\alpha_1, \ldots, \alpha_k \in \N_0$,
    and let $\beta \in \N_0$.
    Let $A := \prod_{i=1}^k \Z_{p^{\alpha_i}}$, and let
    $B$ be an abelian group of exponent $p^{\beta}$,
    and let $f :A \to B$. Then 
    $\fundeg (f) \le \beta \sum_{i=1}^k (p^{\alpha_i} - 1)$.
  \end{lem}
  \begin{proof}
    We first prove the result for the case
    $k = 1$ and $\beta = 1$, and we set $\alpha := \alpha_1$.
    In this case, $A$ is cyclic, and we let $a$
  be a generator of $A$.
  By Lemma~\ref{lem:gen}, it is sufficient to
  show $(\tau_a - 1)^{p^{\alpha}} * f = 0$. We assume $f \neq 0$.
  Let $\Ann (f)$ be the ideal of $\Z[A]$ defined
  by $\Ann(f) := \{r \in \Z[A] \mid r * f  = 0\}$.
  Since $\exp (B) = p$, we have $p \in \Ann (f)$, and therefore
  the quotient ring $\Z[A]/\Ann (f)$ is a commutative ring of characteristic~$p$.
  Hence $(\tau_a - 1)^{p^{\alpha}} \equiv_{\Ann(f)}
  \tau_a^{p^{\alpha}} - 1 =
   \tau_{p^{\alpha} a} - 1 = \tau_0 - 1 = 1 - 1 = 0$. Thus 
   $(\tau_a - 1)^{p^{\alpha}} * f = 0$, finishing the case
   $k = 1$ and $\beta = 1$.

   Next, we consider the case $k = 1$ with $\beta \in \N_0$ arbitrary.
   Again, we set $\alpha := \alpha_1$, and we proceed by induction on $\beta$.
  If $\beta = 0$, then $|B|=1$ and hence $f$ is constant
  and therefore of degree $0$.
  Suppose $\beta \ge 1$. The group $A$ is cyclic.
  Hence by Lemma~\ref{lem:gen}, 
  it is sufficient to
  show $(\tau_a - 1)^{\mu + 1} * f = 0$, where
  $\mu = \beta (p^{\alpha} - 1)$ and $a$ is a generator of the group $A$.
  We let $C := \{b \in B \mid p\, b = 0\}$
  be the subgroup consisting of all elements of order $p$, and let
  $\pi$ be the projection of $B$ to $B/C$. $B/C$ has exponent
  $p^{\beta - 1}$, and
  $\pi \circ f$ is a function from $A$ to $B/C$.
  Therefore, $\fundeg (\pi \circ f) \le (\beta - 1) (p^{\alpha} - 1)$
  by the induction hypothesis. Thus $0 = (\tau_a - 1)^{(\beta - 1) (p^{\alpha} - 1) + 1} * (\pi \circ f) =
  (\tau_a - 1) * ((\tau_a - 1)^{(\beta - 1) (p^{\alpha} - 1)} * (\pi \circ f))$.
  Hence by Lemma~\ref{lem:gen}, 
  $(\tau_a - 1)^{(\beta - 1) (p^{\alpha} - 1)} * (\pi \circ f)$ is constant.
  Since for every $r \in \Z[A]$, $r * (\pi \circ f) =
   \pi \circ (r * f)$,
    this implies that there exists a $b \in B$ and a function
    $g : A \to C$ such that
  \[
  ((\tau_a - 1)^{(\beta - 1) (p^{\alpha} - 1)}) * f \,\, (a) = b + g(a)
  \]
  for all $a \in A$. 
  Since $C$ is of exponent $p$, we can use the case
  $\beta=1$ above to obtain  $(\tau_a - 1)^{p^{\alpha}} * g = 0$.
  Hence, for $e := p^{\alpha} + (\beta - 1) (p^{\alpha} - 1)$,
  we have $(\tau_a - 1)^e * f = 0$. Since $e = \mu + 1$, we
  have completed the proof of $\fundeg (f) \le \mu$ in the case $k=1$ for
  arbitrary $\beta \in \N_0$.

  For the case $k \ge 1$, we observe that by Theorem~\ref{thm:partialtotal},
  $\fundeg(f) \le \sum_{i=1}^k \partdeg_i (f)$.
  From its definition and the case $k=1$ , we see that $\partdeg_i (f)$
  is at most $\beta (p^{\alpha_i} - 1)$, which implies 
  $\fundeg(f) \le \sum_{i=1}^k \beta (p^{\alpha_i} - 1)$.
  \end{proof}

  This upper bound is sometimes actually reached by
  the characteristic function of an element. Although the following
  result also follows from \cite{Ka:TJRO}, we include a direct proof.
  \begin{lem} \label{lem:degchar} \label{lem:dirac}
    Let $p$ be a prime, let $\alpha_1, \ldots, \alpha_k \in \N$,
    let $A := \prod_{i=1}^k \Z_{p^{\alpha_i}}$,
       let $a \in A$, and let
   $\chi_a : A \to \Z_p$ with $\chi_a (a) = 1$ and $\chi_a (x) = 0$
   for $x \neq a$.
   Then $\fdeg (\chi_a) = \sum_{i=1}^k (p^{\alpha_i} - 1)$.
  \end{lem}
  \begin{proof}
    From~\cite[Corollary~2.5]{Ka:TJRO}, we obtain
    $\nu (\Z_p [A]) = 1 + \sum_{i=1}^k (p^{\alpha_i} - 1)$, and
    thus by Lemma~\ref{lem:deltanu}, $\delta (A, \Z_p) =
    \sum_{i=1}^k (p^{\alpha_i} - 1)$. By Lemma~\ref{lem:deltachi},
    $\delta (A, \Z_p) = \fundeg (\chi_0^{A,1}) = \fundeg
    (\tau_{-a} * \chi_0^{A, 1}) = \fundeg (\chi_a)$.

    For a direct proof avoiding Section~\ref{sec:genbounds} and
    \cite{Ka:TJRO},
    we start with the case $k = 1$ and set $\alpha := \alpha_1$.
    In the polynomial ring $\Z_p [t]$, we have
    $(t-1) \cdot (t-1)^{p^{\alpha} - 1} =
     (t-1)^{p^{\alpha}} =
    t^{p^{\alpha}} - 1 = (t-1) \cdot \sum_{i=0}^{p^{\alpha} - 1} t^i$, and
    since $\Z_p [t]$ is an integral domain, we
    obtain $(t-1)^{p^{\alpha} - 1} = \sum_{i=0}^{p^{\alpha} - 1} t^i$.
    Let $\Ann (\chi_a) := \{r \in \Z[A] \mid r*\chi_a \} = 0$.
    Then $p \in \Ann (\chi_a)$, and thus
    $\varphi : \Z_p [t] \to  \Z  [A] / \Ann (\chi_a)$,
    $\sum_{i \in \N_0} (\gamma_i + p\Z) \, t^i \mapsto (\sum_{i \in \N_0} \gamma_i \tau_1^i) + \Ann (\chi_a)$
    is a well defined ring homomorphism.
    Hence we have
    $(\tau_1 - 1)^{p^{\alpha} - 1} * \chi_a = (\sum_{i=0}^{p^{\alpha} - 1} \tau_1^i) * \chi_a$,
    and therefore $((\tau_1 - 1)^{p^{\alpha} - 1} * \chi_a) \, (0)
    = ((\sum_{i=0}^{p^{\alpha} - 1} \tau_1^i) * \chi_a) \, (0) =
       \sum_{i = 0}^{p^{\alpha} - 1} \chi_a (0 + i) = 1$.
       Hence $\fundeg (\chi_a) \ge p^{\alpha} - 1$, and
       thus by Lemma~\ref{lem:pgroupupperbound},
       $\fundeg (\chi_a) = p^{\alpha} - 1$, which finishes the
       case $k=1$.
       For $k \ge 2$, we write $a = (a_1, \ldots, a_k)$ and
       let  $\chi_{a_i} : \Z_{p^{\alpha_i}} \to \Z_p$ be the
       characteristic function of $a_i$.
       Then 
       with the notation of Lemma~\ref{lem:prodprop},
       $\chi_a = \chi_{a_1} \otimes \cdots \otimes \chi_{a_k}$.
       Now this
        Lemma yields
        $\fundeg (\chi_a) = \sum_{i = 1}^k \fundeg (\chi_{a_i})$.
        From the case $k=1$, we infer that the last sum is
        equal to $\sum_{i = 1}^k (p^{\alpha_i}-1)$.
    \end{proof}

  For two finite abelian $p$-groups $A$ and $B$, it would be interesting to
  know the exact value of $\delta (A, B) := \max \, \{ \fundeg (f) \mid f : A \to B\}$.
  By Lemma~\ref{lem:combination} (or Lemma~\ref{lem:deltanu}),
  we may restrict ourselves that
  $B$ is a cyclic group $\Z_{p^\beta}$.
  By Lemma~\ref{lem:deltanu}, finding $\delta (A, \Z_{p^{\beta}}) + 1$
  is equivalent to the following problem:
  \begin{prb} \label{prb:nd}
    For a finite abelian $p$-group $A = \prod_{i=1}^k \Z_{p^{\alpha_i}}$ and $\beta \in \N$, find the nilpotency degree $\nu$ of the augmentation ideal of $\Z_{p^{\beta}} [A]$.
  \end{prb}
  For $\beta = 1$, this problem is solved in Lemmas~\ref{lem:pgroupupperbound}~and~\ref{lem:dirac} (and in \cite{Ka:TJRO}) with the result
  $\nu = 1 + \sum_{i = 1}^k (p^{\alpha_i} - 1)$.
  For $\beta \ge 2$, 
  the considerations at the beginning of Section~\ref{sec:bounds}
  yield
  $\nu \le \beta (1 +(\sum_{i=1}^k (p^{\alpha_i} - 1))$.
  In 
  Lemma~\ref{lem:pgroupupperbound}, this upper bound is lowered
  by $\beta - 1$ to
  $\nu \le 1 + \beta \sum_{i=1}^k (p^{\alpha_i} - 1)$.
    In the case that
    $A = \Z_{p^{\alpha}}$ is cyclic (this is the case $k=1$, and we set $\alpha := \alpha_1$),
    $\nu$ can be computed inside $\Z[x]$ as
    the smallest number such that
    $(x-1)^\nu$ lies in  $(p^\beta, x^{p^{\alpha}} - 1)$.
    From a few experiments, we hypothesize that
    then $\nu = \beta p^{\alpha} - (\beta - 1) p^{\alpha-1}$,
    which would tell
    $\delta (\Z_{p^{\alpha}}, \Z_{p^{\beta}}) =
    \beta p^{\alpha} - (\beta - 1) p^{\alpha-1} -1 $.
    For noncyclic $A$ and $\beta \ge 2$, we note that,
    unlike the case $\beta = 1$, 
    we cannot reduce the problem from arbitrary $A$ to cyclic $A$ because
    using Lemma~\ref{lem:prodprop} requires that $\Z_{p^{\beta}}$ is
    a field, i.e., $\beta = 1$.

  \section{Functions of finite degree} \label{sec:findeg}
  If $A, B$ are periodic abelian groups that are not both $p$-groups for
  the same $p$, then by Theorem~\ref{thm:deltafinite} there exist
  functions from $A$ to $B$ of infinite degree.
  Still, there are functions of finite degree, and they
  have the interesting property that they can be decomposed
  into functions on each $p$-component of $A$ and $B$.
  \begin{lem} \label{lem:fdegsplits}
    Let $K \in \N$, and let $p_1,\ldots,p_K$ be
    pairwise distinct primes.
    For each $i \in \ul{K}$, let $A_i$ and
    $B_i$ be abelian $p_i$-groups.
    Let $A := \prod_{i=1}^K A_i$ and
    $B := \prod_{i=1}^K B_i$, and let $f : A \to B$,
    $f (\vb{a}) = (f_1 (\vb{a}),\ldots, f_K (\vb{a}))$ 
    be a function of finite degree.
    Then there is a family of
    functions $(g_i)_{i \in \ul{K}}$ with
    $g_i : A_i \to B_i$ such that
    \begin{equation} \label{eq:fg}
    f (a_1,\ldots, a_K) =
    (g_1 (a_1), \ldots, g_K (a_K))
    \end{equation}
    for
    all $(a_1,\ldots, a_K) \in A$.
  \end{lem}
  \begin{proof}
    We show that for all $i, j \in \ul{K}$ with $i \neq j$, $f_i$
    does not depend on its $j$th argument.
    To this end, we show that
    \begin{equation} \label{eq:fi}
      f_i (a_1,\ldots, a_K) =
      f_i (a_1,\ldots, a_{j-1}, b, a_{j+1}, \ldots, a_K)
     \end{equation}
    for all $(a_1,\ldots, a_K) \in A$ and $b \in A_j$.
    We fix $(a_1,\ldots, a_K) \in A$ and $b \in A_j$,
    and define
    the functions
    $\alpha_j : A_j \to A$,
    $x \mapsto (a_1,\ldots, a_{j-1}, x, a_{j+1}, \ldots, a_K)$,
    and $\pi_i : B \to B_i$, $(b_1,\ldots, b_K) \mapsto
    b_i$. The function $\alpha_j$ is of degree at most~$1$ by Lemma~\ref{lem:combination} and items~\eqref{it:dp3} and \eqref{it:dphom} of Lemma~\ref{lem:dprop}.
    The degree of $\pi_i$ is at most $1$
    by Lemma~\ref{lem:dprop}\eqref{it:dphom}.
    Hence by Theorem~\ref{thm:comp},
    the degree of $h := \pi_i \circ f \circ \alpha_j$ is
    finite. Since $h : A_j \to B_i$,
    Lemma~\ref{lem:pq} implies that $h$ is constant.
    Hence $h(a_j) = h(b)$, which implies~\eqref{eq:fi}.
    Thus $f_i$ does not depend on its $j$\,th argument.
    This implies that the function $f$ can be
    written in the form given in~\eqref{eq:fg}.
  \end{proof}

  Hence for finite abelian groups $A, B$, we can explicitly
  give an
  $n \in \N$ (depending on $A$ and $B$)
  such that $\fundeg (g) \le n$ or $\fundeg (g) = \infty$
  for all $g : A \to B$.
    \begin{thm}
     Let $K \in \N$, and let $p_1,\ldots,p_K$ be
    pairwise distinct primes.
    Let $A := \prod_{i=1}^K \prod_{j=1}^{L_i} \Z_{p_i^{\alpha (i,j)}}$
    and $B := \prod_{i=1}^K \prod_{j=1}^{M_i} \Z_{p_i^{\beta (i,j)}}$,
    and let $f:A \to B$.
    Then the following are equivalent:
    \begin{enumerate}
    \item \label{it:fd1} $f$ is of finite degree.
    \item \label{it:fd2} For all $i_1, i_2 \in \ul{K}$, $j_1 \in \ul{L_{i_1}}$,
      and $j_2 \in \ul{M_{i_2}}$ we have: if $i_1 \neq i_2$, then
      the $(i_2, j_2)$-th component of the 
      value
      $f (a)$ does not depend on the
      $(i_1, j_1)$-th component of $a$.
    \item \label{it:fd3} $\fundeg (f) \le \max_{i \in \ul{K}}
      (\,
        (\max_{j \in \ul{M_i}} (\beta (i,j))) \, \cdot \,
        (\sum_{j=1}^{L_i} (p^{\alpha (i,j)} - 1))
      \,)$.
    \end{enumerate}
  \end{thm}
  \begin{proof}
    \eqref{it:fd1}$\Rightarrow$\eqref{it:fd2}:
    For $\vb{a} \in A$, let $f_{i,j} (\vb{a}) \in \Z_{p_i^{\beta(i,j)}}$
    be the $(i,j)$th component of $f(a)$.
    Writing $\vb{a} = (a_1, \ldots a_K)$ with
    $a_i \in \prod_{j=1}^{L_i} \Z_{p_i^{\alpha (i,j)}}$, Lemma~\ref{lem:fdegsplits}
        yields that for $i_1 \neq i_2$ and $j_2 \in \ul{M_{i_2}}$, $f_{i_2, j_2}$
        does not depend on $a_{i_1}$.
        
   \eqref{it:fd2}$\Rightarrow$\eqref{it:fd3}:
   With $A_i := \prod_{j=1}^{L_i} \Z_{p_i^{\alpha (i,j)}}$
   and $B_i := \prod_{j=1}^{M_i} \Z_{p_i^{\beta (i,j)}}$,
   condition~\eqref{it:fd2} tells that we can write 
   $f(\vb{a}) = f(a_1, \ldots, a_K) = (g_1(a_1), \ldots, g_K(a_K))$.
   Let $g_i' (a_1,\ldots, a_K) := g_i (a_i)$.
   Now by Lemma~\ref{lem:combination},
   $\fundeg(f) \le \max_{i \in \ul{K}} (\fundeg (g_i'))$.
   Let us now fix $i \in \ul{K}$.
   By Theorem~\ref{thm:gfs}, we have
   $\fundeg (g_i') \le \partdeg_i (g_i')$, which is equal to
   $\fundeg (g_i)$ by the definition of the partial degree.
   With $g_i (a_i) = (g_{i,1} (a_i), \ldots, g_{i, M_i} (a_i))$,
   Lemma~\ref{lem:combination} yields
   $\fundeg (g_i) = \max_{j \in \ul{M_i}} (\fundeg (g_{i,j}))$.
   The function $g_{i,j}$ maps $\prod_{r=1}^{L_i} \Z_{p_i^{\alpha (i,r)}}$
   into $\Z_{p_i^{\beta (i,j)}}$. Form Lemma~\ref{lem:pgroupupperbound},
   we obtain $\fundeg (g_{i, j}) \le \beta(i,j) \sum_{r=1}^{L_i} (
   p_i^{\alpha (i,r)} - 1)$.
   Thus $\fundeg (g_i) = \max_{j \in \ul{M_i}} (\fundeg (g_{i,j}))
   \le                   \max_{j \in \ul{M_i}} (
                        \beta(i,j) \sum_{r=1}^{L_i} (
                        p_i^{\alpha (i,r)} - 1))
                        =
                        (\max_{j \in \ul{M_i}} (
                        \beta(i,j))) \cdot (\sum_{j=1}^{L_i} (
                        p_i^{\alpha (i,j)} - 1))$.
 
   \eqref{it:fd3}$\Rightarrow$\eqref{it:fd1}: immediate.
   \end{proof}                     

   \section{Chevalley's Theorem}
   The properties of the functional degree developed so far
   allow us to put N.\ Alon's proof of Chevalley's Theorem
   \cite[Theorem~3.1]{Al:CN} into a more general frame.
   For functions $f_1, \ldots, f_r : A^N \to B$, we define
   \[
     V(f_1, \ldots, f_r) := \{ \vb{a} \in A^N \mid
     f_1 (\vb{a}) = \cdots = f_r (\vb{a}) = 0 \}
     \]
   to be the set of common zeroes of the $f_i$'s.
    
   \begin{thm} \label{thm:chevgroup}
    Let $p$ be a prime, let $m, n \in \N$, 
    let $\alpha_1, \ldots, \alpha_m,
    \beta_1, \ldots, \beta_n \in \N_0$, and
    let 
    $A = \prod_{i \in \ul{m}} \Z_{p^{\alpha_i}}$ and
    $B = \prod_{i \in \ul{n}} \Z_{p^{\beta_i}}$ be two finite abelian $p$-groups.
    Let $f_1, \ldots, f_r : A^N \to B$ be functions, and assume
    \[
   N \,  (\sum_{i = 1}^m (p^{\alpha_i} - 1)) > (\sum_{i = 1}^r \fundeg (f_i)) \, (\sum_{i = 1}^n (p^{\beta_i} - 1)).
    \]
    Then $V(f_1, \ldots ,f_r)$ is not a singleton.
  \end{thm} 
   \begin{proof}
     Seeking a contradiction, we suppose that there is
     $\vb{a} \in A^N$ with
     $V(f_1, \ldots ,f_r) = \{ \vb{a} \}$. 
     Let $\chi_{\vb{a}}^{A^N} : A^N \to \Z_p$ be the characteristic function
     of $\vb{a}$ defined by $\chi_{\vb{a}} (\vb{x}) = 1$ if $\vb{x} = \vb{a}$ and
     $\chi_{\vb{a}} (\vb{x}) = 0$ else.
     Similarly, we define the characteristic function $\chi_{0}^{B^r} : B^r \to \Z_p$
     of $0$ in $B^r$.
     From the assumption $V(f_1, \ldots, f_r) = \{ \vb{a} \}$,
     we obtain
     \[
     \chi_{0}^{B^r} (f_1, \ldots, f_r) = \chi_{\vb{a}}^{A^N}.
     \]
     By Lemma~\ref{lem:degchar}, which we apply with
     $k := Nm$, we have $\fundeg (\chi_{\vb{a}}^{A^N})
     = N (\sum_{i=1}^m (p^{\alpha_i} - 1))$.
     Hence $N (\sum_{i=1}^m (p^{\alpha_i} - 1)) =
     \fundeg (\chi_{0}^{B^r} (f_1, \ldots, f_r))$.
     By Theorem~\ref{thm:gfs}, we have
     $\fundeg (\chi_{0}^{B^r} (f_1, \ldots, f_r)) \le
     \sum_{i=1}^r \partdeg_i (\chi_{0}^{B^r}) \cdot \fundeg (f_i)$.
     Since for each $i \in \ul{r}$, $\partdeg_i (\chi_{0}^{B^r})$
     is the supremum of degrees of functions from $B$ to $\Z_p$, 
     Lemma~\ref{lem:pgroupupperbound} yields
     $\partdeg_i (\chi_{0}^{B^r}) \le \sum_{j = 1}^n (p^{\beta_j} - 1)$.
     Thus
     $N (\sum_{i=1}^m (p^{\alpha_i} - 1)) \le (\sum_{i=1}^r \fundeg (f_i)) \, (\sum_{j = 1}^n (p^{\beta_j} - 1))$,
     contradicting the assumption.
   \end{proof}
   We note
   that two of the sums occurring in this theorem come
   from our computation of $\delta(A, \Z_p)$ and
   $\delta (B, \Z_p)$ in Sections~\ref{sec:genbounds}~and~\ref{sec:bounds}.
   Setting $A = B$, we obtain the following corollary.
   \begin{thm} \label{thm:chevgroupA}
     Let $p$ be a prime, let
         $A$ be a finite abelian $p$-group with $|A| > 1$, and
    let $f_1, \ldots, f_r : A^N \to A$ be functions with
    \[
   N  > \sum_{i = 1}^r \fundeg (f_i).
    \]
    Then $V(f_1, \ldots ,f_r)$ is not a singleton.
   \end{thm}
   With a little more effort, one also obtains
   a new version of Warning's First Theorem.
    
   \section{Warning's First Theorem} \label{sec:warning1}
   Similar to the well known proof of
   Warning's First Theorem (cf. \cite{Ax:ZOPO, As:ANPO}),
   our proof relies on a fact on the sum of all values
   of a function of nonmaximal degree.
   We recall from Section~\ref{sec:genbounds} that for a finite abelian group
   $A$, the
   characteristic functions $\chi_{a}^{A,1}$ from $A$ to $\Z_p$
    attain the  maximal degree $\delta (A, \Z_p)$.
    \begin{lem} \label{lem:sum0}
      Let $p$ be a prime, let $A$ be a finite abelian $p$-group, let
      $\chi_0 = \chi_0^{A,1} : A \to \Z_p$ be the characteristic function of $0$,
      let $M := \fundeg (\chi_0)$, and let $f : A \to \Z_p$.
      If $\fundeg (f) < M$, then $\sum_{x \in A} f(x) = 0$.
    \end{lem}
    \begin{proof}
      Every $g \in \Z_p^A$ satisfies
      \begin{equation} \label{eq:g1}
        g = \sum_{x \in A} g (x) (\tau_{-x} * \chi_0),
      \end{equation}
      and therefore (or by Lemma~\ref{lem:deltachi})
      $\fundeg (g) \le M$.  
      Let $I := \AugZA$. Clearly,
      $I * \chi_0 = \{ i * \chi_0 \mid i \in I\}$ is a subvectorspace
      of $\Z_p^A$. Let $D$ be the subvectorspace of
      $\Z_p^A$ generated by $\chi_0$.
      We will first prove
      \begin{equation} \label{eq:vspace}
      I * \chi_0 + D = \Z_p^A.
      \end{equation}
      Let $h \in \Z_p^A$. From~\eqref{eq:g1}, we obtain
      $r \in \Z[A]$ with $r * \chi_0 = h$. Now take
      $z \in \Z$ such  that $r - z \tau_0 \in \AugZA$.
      Let $i := r - z \tau_0$.
      Then $h = r * \chi_0 = i * \chi_0 + (z \tau_0)* \chi_0 =
      i * \chi_0 + z \chi_0$.
      Hence $h \in I * \chi_0 + D$, which concludes the proof
      of~\eqref{eq:vspace}.
      Let \[S := \{ f \in \Z_p^A \mid \fundeg (f) < M\}. \]
      Then $S$ is a subvectorspace of $\Z_p^A$, and since
      $\chi_0 \not\in S$, we have $S \neq \Z_p^A$.
      By Lemma~\ref{lem:dprop}\eqref{it:dp2}, we have
      $I * \chi_0 \subseteq S$. By~\eqref{eq:vspace},
      $I * \chi_0$ is of codimension at most $1$, and therefore
      $I * \chi_0 = S$. Now by the assumptions $f \in S$,
      and therefore $f \in I * \chi_0$, and thus
      we have $i \in I$ with
      $f = i * \chi_0$.
      We can write $i = \sum_{a \in A} (\tau_a - 1) r_a$ with
      $r_a \in \Z[A]$ for each $a \in A$.
      Now
      \[
        \begin{split}
          \sum_{x \in A} f(x) &=
      \sum_{x \in A} \sum_{a \in A} (((\tau_a - 1) r_a * \chi_0) \, (x)) \\ &=
      \sum_{a \in A} \sum_{x \in A} (((\tau_a - 1) r_a * \chi_0) \, (x)).
        \end{split}
      \]  
      We will now show that for each $a \in A$, the corresponding summand
      is $0$.
      To this end, we compute
      \[
      \begin{split}
        \sum_{x \in A} (((\tau_a - 1) r_a * \chi_0) \, (x)) &=
           \sum_{x \in A} (\tau_a * (r_a * \chi_0)) \, (x) -
           \sum_{x \in A} (r_a * \chi_0) \, (x) \\
           &= \sum_{x \in A} (r_a * \chi_0) \, (x + a) -
           \sum_{x \in A} (r_a * \chi_0) \, (x).
      \end{split}     
      \]
      Since the mapping $x \mapsto x + a$ is a bijection on the group $A$,
      we have $\sum_{x \in A} (r_a * \chi_0) \, (x + a) =
               \sum_{x \in A} (r_a * \chi_0) \, (x)$.
               Hence the summand corresponding to $a$ is $0$, which proves
               $\sum_{x \in A} f(x) = 0$.
  \end{proof}             
    Now we can improve Theorems~\ref{thm:chevgroup} and~\ref{thm:chevgroupA}.
    \begin{thm} \label{thm:warngroup}
    Let $p$ be a prime, let $m, n \in \N$, 
    let $\alpha_1, \ldots, \alpha_m,
    \beta_1, \ldots, \beta_n \in \N_0$, and
    let 
    $A = \prod_{i \in \ul{m}} \Z_{p^{\alpha_i}}$ and
    $B = \prod_{i \in \ul{n}} \Z_{p^{\beta_i}}$ be two finite abelian $p$-groups.
    Let $f_1, \ldots, f_r : A^N \to B$ be functions, and assume
    \[
   N \,  (\sum_{i = 1}^m (p^{\alpha_i} - 1)) > (\sum_{i = 1}^r \fundeg (f_i)) \, (\sum_{i = 1}^n (p^{\beta_i} - 1)).
    \]
    Then $p$ divides $|V(f_1, \ldots ,f_r)|$.
   \end{thm} 
     \begin{proof}
       Let $\chi_{0}^{B^r} : B^r \to \Z_p$
       be the characteristic function of $0$ in $B^r$.
       Then
       \begin{equation} \label{eq:thesum1}
       \sum_{\vb{x} \in A^N} \chi_{0}^{B^r} (f_1, \ldots, f_r) \, (\vb{x})
       \end{equation}
       is the congruence class of $|V(f_1, \ldots, f_r)|$ modulo $p$.
       It is therefore sufficient to prove that~\eqref{eq:thesum1}
       is $0$.
       Let $M$ be the functional degree of the characteristic function
       $\chi_{0}^{A^N} : A^N \to \Z_p$.
       From Lemma~\ref{lem:degchar}, we obtain
       $M = N \sum_{i=1}^m (p^{\alpha_i} - 1)$.
       We have
       $\fundeg (\chi_{0}^{B^r} (f_1, \ldots, f_r)) =
       \fundeg (\prod_{i=1}^r \chi_0^{B} \circ f_i)$,
       where $\chi_0^B : B \to \Z_p$ is the characteristic function
       of $0$ on $B$.
       By Lemma~\ref{lem:mprop}, $\fundeg (\prod_{i=1}^r \chi_0^{B} \circ f_i)
       \le
       \sum_{i=1}^r \fundeg (\chi_0^{B}) \fundeg (f_i)$,
       which by Lemma~\ref{lem:pgroupupperbound} is 
       $\le \sum_{i=1}^r \fundeg (f_i) \sum_{j = 1}^n (p^{\beta_j} - 1))$.
       By the assumption, this last expression is less than
       $ N \,  (\sum_{i = 1}^m (p^{\alpha_i} - 1)) = M$.
       Now Lemma~\ref{lem:sum0} yields
       $\sum_{\vb{x} \in A^N} \chi_{0}^{B^r} (f_1, \ldots, f_r) \, (\vb{x}) = 0$.
     \end{proof}
     Again by setting $A = B$, we obtain:
     \begin{thm} \label{thm:warngroupA}
     Let $p$ be a prime, let
         $A$ be a finite abelian $p$-group with $|A| > 1$, and
    let $f_1, \ldots, f_r : A^N \to A$ be functions with
    \[
   N  > \sum_{i = 1}^r \fundeg (f_i).
    \]
    Then $p$ divides $|V(f_1, \ldots ,f_r)|$.
   \end{thm}

  We apply this Theorem to polynomial functions over not necessarily
  commutative rings. For such a ring $R$
  and $X = \{x_1, \ldots, x_n\}$, a \emph{monomial}
  is a nonempty word over the alphabet $R \cup X$.
  We denote the set of nonempty words over the alphabet $R \cup X$
  by $(R \cup X)^+$.
  A \emph{polynomial
    expression over $R$ in the variables $X$} is a sum $\sum_{m \in (R \cup X)^+} z_m m$ with $z_m \in \Z$
  and only finitely many $z_m \neq 0$.
  The \emph{degree} of a monomial $y_1y_2\ldots y_k$ is defined
  as $\# \{i \in \ul{k} \mid y_i \in X \}$.
    The degree
  of $\sum_{m \in (R \cup X)^+} z_m m$ is the  defined as the maximal degree of
  those monomials $m$ with $z_m \neq 0$; we set $\deg (0) := 0$. 
  As an example, let $R$ be the ring $\matrixring{2}{\Z}$  of $2 \times 2$-matrices over
  $\Z$. Then the degree of the polynomial expression
  $g = 5\, x_1 \MatTwo{1}{-2}{3}{5} x_1 x_2 \MatTwo{1}{0}{0}{1} x_2 +
  0 \, x_1 x_1 x_2 x_3 \MatTwo{1}{0}{0}{-1}  x_7 +
  2 \, x_1 \MatTwo{2}{8}{7}{6}$ is $\max \, (4, 1) = 4$.
    For a polynomial expression $f$, we write $\deg (f)$ for its degree.
  A polynomial expression $f$  in the variables $x_1,\ldots,x_n$ induces
  a function $\ol{f} : R^n \to R$; in the example above,
  the function $\ol{g}$ induced
  by $g$ is given by $\ol{g}(x_1, \ldots, x_7) = 5 x_1 \MatTwo{1}{-2}{3}{5} x_1 x_2^2 +
  2 \, x_1 \MatTwo{2}{8}{7}{6}$ for all $\vb{x} \in \matrixring{2}{\Z}^7$.
  \begin{lem} \label{lem:syntdeg}
    Let $R$ be a ring, let $f$ be a polynomial expression over $R$, and
    let $\ol{f}$ be the function that $f$ induces in $R$.
    Then $\fundeg (\ol{f}) \le \deg (f)$.
  \end{lem}
  \begin{proof}
  We assume $|R| > 1$. Let $f = \sum_{m \in I} z_m m$ for some finite set
  of monomials $I$ over $R \cup \{x_1,\ldots, x_n\}$. We assume $I \neq \emptyset$
  and that $z_m \neq 0$ for all $m \in I$. 
  The functional degree of the function induced
  by a variable $x_i$ is $1$, and the functional degree of a constant function
  is $0$. Lemma~\ref{lem:mprop} yields that $\fundeg (\ol{m}) \le \deg (m)$
  for every $m \in I$.
  Now Lemma~\ref{lem:dprop}\eqref{it:dphom} and Theorem~\ref{thm:comp} yield
  $\fundeg (z_m \, \ol{m}) \le \fundeg(\ol{m}) \le \deg(m)$.
  Applying Lemma~\ref{lem:aprop}, we obtain
  $\fundeg (\ol{f}) = \fundeg (\sum_{m \in I} z_m \, \ol{m}) \le
  \max_{m \in I} \fundeg (z_m \ol{m}) \le
  \max_{m \in I} \deg (m) = \deg (f)$.
  \end{proof}
  
  \begin{thm} \label{thm:warnnoncommring}
    Let $p$ be a prime, let $\alpha \in \N$, 
    let $R$ be a (not necessarily commutative) finite ring with $|R| = p^{\alpha}$, let $N \in \N$, let
    $X = \{x_1, \ldots, x_N\}$, and let $f_1, \ldots, f_r$ be polynomial
    expressions over $R$ in the variables $X$.
    If
    \(
    N > \sum_{i = 1}^r \deg (f_i),
    \)
    then 
    $p$ divides $|V(f_1, \ldots ,f_r)|$.
  \end{thm}
  \begin{proof}
    We want to apply Theorem~\ref{thm:warngroupA} setting $A$
    to be the
    additive group of $R$.
    By the assumption, we have
    $N > \sum_{i = 1}^r \deg (f_i)$, and thus by Lemma~\ref{lem:syntdeg},
    $N > \sum_{i = 1}^r \fundeg (\ol{f_i})$.
    Now Theorem~\ref{thm:warngroupA} yields that $p$ divides
    $|V(f_1, \ldots ,f_r)|$.
  \end{proof}     
    
\section{The functional degree of polynomial functions}
  In this section, we will compute the functional
  degree of polynomial functions on a field $F$.
  We denote the \emph{total degree} of a polynomial $f \in F[x_1, \ldots, x_n]$
by $\deg (f)$; here, for $c \in F \setminus \{0\}$,
the total degree of a monomial
$m = c \prod_{i=1}^n x_i^{\alpha_i}$ is defined by
$\deg (m) := \sum_{i = 1}^n \alpha_i$; the total degree of a polynomial
is the maximum of the total degrees of its monomials, and $\deg (0)$
is additionally set to $0$.
For a field of characteristic $p$, we will also use the
\emph{$p$-weight degree} that was defined in \cite{MM:IOTC}:
By $s_p (n)$, we denote the digit sum of $n$ in base $p$.
For $c \in F \setminus \{0\}$ and
$\alpha_1, \ldots, \alpha_n \in \N_0$, the
the $p$-weight degree of a monomial $m = c \prod_{i=1}^n x_i^{\alpha_i}$
is defined
by
\[
\pdeg (c \prod_{i=1}^n x_i^{\alpha_i}) := \sum_{i=1}^n s_p (\alpha_i).
\]
The $p$-weight degree of a polynomial is the maximum of the $p$-weight degrees
of its monomials, and we set $\pdeg (0) := 0$.
In Theorem~\ref{thm:fp}, we will see that a
polynomial $f$ over a finite field $F$ which is reduced
modulo all $x_i^{|F|} - x_i$ induces a function of functional
degree $\pdeg (f)$. 
    \begin{lem} \label{lem:xn}
      Let $p$ be a prime, and
      let 
    $F$ be a field of characteristic $p$. For each $n \in \N_0$ and $x \in F$,
    let $f_n (x) := x^n$.
    Then \begin{enumerate}
          \item \label{it:w1} $\fundeg (f_n) \le s_p (n)$;
          \item \label{it:w2} If $F$ is finite and $n < |F|$, then $\fundeg (f_n) = s_p (n)$.
         \end{enumerate}
  \end{lem}
  \emph{Proof:}
  \eqref{it:w1}
  We let $n \in \N_0$, and       
  let $t \in \N$ and  $\gamma_0, \ldots, \gamma_{t-1} \in \{0, \ldots, p-1 \}$ be
  such that $n = \sum_{i=0}^{t - 1} \gamma_i \, p^i$.
  Let $i \in \{0,\ldots, t - 1\}$.
  The functional degree of $f_1(x) = x$ is equal to $1$. Hence
  by Lemma~\ref{lem:mprop}, the functional degree of
  $f_{\gamma_i} (x) = x^{\gamma_i}$ is at most $\gamma_i$. Since
  $h_i : F \to F$, $h_i(x) := x^{p^i}$ is a group homomorphism,
  Lemma~\ref{lem:dprop}\eqref{it:dphom}
  and Theorem~\ref{thm:comp} imply $\fundeg (f_{\gamma_i \, p^i})
  = \fundeg (h_i \circ f_{\gamma_i}) \le
  \fundeg (f_{\gamma_i}) \le \gamma_i$.
  Since $f_n = \prod_{i=0}^{t-1} h_i \circ f_{\gamma_i}$, 
  Lemma~\ref{lem:mprop}  implies
  $\fundeg (f_n) \le \sum_{i = 0}^{t-1} \gamma_i = s_p (n)$.
  This completes the proof of~\eqref{it:w1}.

  \eqref{it:w2}: Let $q := |F|$, and let $\alpha \in \N$ be such
  that $p^{\alpha} = q$. We assume that $n \le q-1$ and
  $n = \sum_{i=0}^{\alpha-1} \gamma_i \, p^i$. We set
   $k := \sum_{i=0}^{\alpha - 1} (p-1-\gamma_i)\, p^i$.
  Let $g (x) := 1 - x^{q-1}$. Since the additive group of $F$ is
  isomorphic to $\Z_p^{\alpha}$, the function $g$ can also be seen
  as a function from $\Z_p^{\alpha}$ to $\Z_p$, and it satisfies
  the assumptions on $f$ in Lemma~\ref{lem:dirac}. By this Lemma,
  we have
  $\fundeg (g) = \alpha (p-1)$.
  Now Lemma~\ref{lem:aprop}\eqref{it:d2} yields
  $\fundeg (g) = \fundeg (f_{q-1}) =
   \fundeg ( f_k \cdot  f_n)$.
   By Lemma~\ref{lem:mprop} and item~\eqref{it:w1},
   we have
   $\fundeg (f_k \cdot f_n) \le
   \fundeg (f_k) + \fundeg (f_n) \le
   \sum_{i=0}^{\alpha-1} (p-1-\gamma_i) + \fundeg (f_n)
   \le
   \sum_{i=0}^{\alpha-1} (p-1-\gamma_i) + \sum_{i=0}^{\alpha-1} \gamma_i
   =
    \alpha (p-1)$.
   Hence all inequalities in this chain are equalities, 
   and so $\fundeg (f_n) = \sum_{i = 0}^{\alpha-1} \gamma_i =
   s_p (n)$.
  \qed

 \begin{lem} \label{lem:mondeg}
    Let $p$ be a prime number, let $\alpha, n \in \N$, let
    $F$ be a field of chacteristic $p$, and let
    $e_1,\ldots,e_n \in \N_0$.
    Let $f (x_1,\ldots,x_n) := x_1^{e_1} \cdots x_n^{e_n}$
    for $\vb{x} \in F^n$.
    Then \begin{enumerate}
    \item $\fundeg (f) \le  \sum_{i=1}^n s_p (e_i)$.
    \item \label{it:pd2} If $|F|$ is finite and $e_1, \ldots, e_n \in \{0, \ldots, |F|-1\}$, then
      $\fundeg (f)  =  \sum_{i=1}^n s_p (e_i)$.
    \end{enumerate}  
  \end{lem}
  \emph{Proof:}
  The claim follows from  Lemma~\ref{lem:prodprop} and
  Lemma~\ref{lem:xn}.
    \qed
  
    \begin{thm} \label{thm:fp}
     Let $p$ be a prime number,  let
     $F$ be a finite field with $q= p^{\alpha}$ elements, let $n \in \N$,
    let $f = \sum_{\vb{e} \in E} c_\vb{e} x_1^{e_1} \dots x_n^{e_n}
      \in F[x_1, \ldots, x_n]$, where $E$ is a finite subset of $\N_0^n$ and
      $(c_\vb{e})_{\vb{e} \in E}$ is a family from $F \setminus \{0\}$,
    and let $\ol{f}$ be the
    function from $F^n$ to $F$ induced by $f$. Then
    $\fundeg (\ol{f})  \le  \pdeg (f)$.
    If all exponents occurring in $f$ are at most $q-1$, i.e,
    if $E \subseteq \{0,1,\ldots, q-1\}^n$,
        then $\fundeg (\ol{f}) = \pdeg (f)$.
  \end{thm}
  \begin{proof}
   If $E = \emptyset$, then $f$ is the $0$-polynomial, and
  we have $\pdeg (f) = 0$ and $\fundeg (\ol{f}) = 0$.
   
   Let us now assume $f \neq 0$.
   By the definition of the $p$-weight degree, we have
   $\pdeg (f) = \max \, \{ \sum_{i = 1}^n s_p (e_i) \mid (e_1, \ldots, e_n) \in E \}$. 
   For each $\vb{e} \in E$, Lemma~\ref{lem:mondeg} and
   Lemma~\ref{lem:mprop} yield that
  the functional degree of the function $\ol{m}$ induced by %
  $m = c_{\vb{e}} x_1^{e_1} \cdots x_n^{e_n}$
  satisfies  $\fundeg (\ol{m}) \le \sum_{i=1}^n s_p (e_i)$.
  Lemma~\ref{lem:aprop}\eqref{it:d1} now implies $\fundeg (\ol{f}) \le \pdeg (f)$.

  For proving the claimed equality, we may assume $f \neq 0$.
  Let $c_{\vb{a}} x_1^{a_1} \dots x_n^{a_n}$ be a monomial
  of maximal $p$-weight degree in $f$.
  Let $g$ be the remainder of $x_1^{q-1-a_1} \cdots x_n^{q-1-a_n} \cdot f$
  modulo $x_1^q - x_1, \ldots, x_n^q - x_n$. 
  We claim that $g$ contains the monomial
  $c_{\vb{a}} x_1^{q-1} \cdots x_n^{q-1}$.
  To this end, let $c_{\vb{b}} x_1^{b_1} \cdots x_n^{b_n}$ be a
  monomial in $f$ such that
  the remainder of $(x_1^{q-1-a_1} \cdots x_n^{q-1-a_n})
   \cdot (x_1^{b_1} \cdots x_n^{b_n})$ modulo
  $x_1^q - x_1, \ldots, x_n^q - x_n$ is $x_1^{q-1} \cdots x_n^{q-1}$.
   Then for each $i \in \ul{n}$, $q-1 - a_i + b_i$ is
   either $q-1$ or $2 (q-1)$, and therefore $b_i = a_i$
   or ($a_i = 0$ and $b_i = q-1$).
   Since $\sum_{i = 1}^n s_p (b_i) \le \sum_{i=1} s_p (a_i)$,
   the alternative $a_i = 0$ and $b_i = q-1$ may never occur.
   Therefore, $\vb{a} = \vb{b}$, and hence
   only the monomial $c_{\vb{a}} x_1^{a_1} \cdots x_n^{a_n}$ from $f$ 
   contributes to the monomial $x_1^{q-1} \cdots x_n^{q-1}$
   in $g$.
   Now by Lemma~\ref{lem:mondeg}\eqref{it:pd2}, this monomial 
   induces a function of functional degree $n \alpha (p-1)$,
   and all other monomials in $g$ induce a function of functional
   degree at most $n \alpha (p-1) - 1$. Lemma~\ref{lem:aprop}\eqref{it:d2}
   implies $\fundeg (\ol{g}) =  n \alpha (p-1)$.
   Thus $n \alpha (p-1) = \fundeg (\ol{g})
          \le \fundeg (x_1^{q-1-a_1} \cdots x_n^{q-1-a_n}) +
          \fundeg (\ol{f}) =
          \sum_{i=1}^n s_p (q-1-a_i) + \fundeg (\ol{f})$.
          Since $s_p ((q-1) - z) = \alpha (p-1) - s_p (z)$ for all
          $z \in \{0, \ldots, q-1\}$, we have
          $\sum_{i=1}^n s_p (q-1-a_i) + \fundeg (\ol{f}) =
          n \alpha (p-1) - \sum_{i = 1}^n s_p (a_i)
          + \fundeg (\ol{f})$. From this chain of inequalities,
          we obtain
  $\fundeg (\ol{f}) \ge \sum_{i = 1}^n s_p (a_i) =
                   \pdeg (f)$.
  \end{proof}

  Now we can 
  derive a special case of Moreno and Moreno's improvement
  \cite[Theorem~1]{MM:IOTC} of
  Warning's First Theorem.
  \begin{thm}[cf. \cite{MM:IOTC}] \label{thm:warnA}
   Let $p$ be a prime, let $F$ be a finite field of characteristic $p$, let $r, N \in \N$, and let
   $f_1, \ldots, f_r \in F[x_1, \ldots, x_N]$.
   We assume that $N > \sum_{j = 1}^r \deg_p (f_j)$.
   Then $p$ divides $|V(f_1,\ldots, f_r)|$.
  \end{thm}
  \begin{proof}
    The additive group of $F$ is a finite abelian $p$-group.
    From the assumption and Theorem~\ref{thm:fp}, we obtain
    $N > \sum_{j=1}^r \fundeg (\ol{f_j})$. Now Theorem~\ref{thm:warngroupA}
    yields that $p$ divides $|V(f_1,\ldots, f_r)|$.
  \end{proof}  

  \section{Chevalley Warning Theorems with restricted domain and range} \label{sec:chevrest}
  We start from a variant of Chevalley's Theorem that was
  formulated and proved
  in \cite{Br:CTWR} in the following form:
    \begin{thm}[{\cite[Theorem~1]{Br:CTWR}}] \label{thm:rest0}
      Let $F$ be a finite field with $q$ elements, let
      $f_1, \ldots, f_r \in F[x_1, \ldots, x_N]$, and let
      $A_1, \ldots, A_N$ be non-empty subsets of $F$ with
      \[
      \sum_{i=1}^N (|A_i|- 1) > (q-1) \sum_{j=1}^r \deg (f_j).
      \]
      Then the set $\{ \vb{a} \in \prod_{i=1}^N A_i \mid
      f_1 (\vb{a}) = \cdots =  f_r (\vb{a}) = 0 \}$
      is not a singleton. 
    \end{thm}
    
    In the case that all subsets $A_i$ are subgroups of the additive
    group of $F$, we can sometimes improve this result:
    \begin{thm} \label{thm:rest1}
       Let $p$ be a prime, $\alpha \in \N$, and 
       let $F$ be a finite field with $q = p^{\alpha}$ elements.
       Let $f_1, \ldots, f_r \in F[x_1, \ldots, x_N]$, let
       $A_1, \ldots, A_N$ be subgroups of $\algop{F}{+}$ with
       $|A_i| = p^{\alpha_i}$ for $i \in \{1, \ldots, N\}$.
       We assume that
      \[
      \sum_{i=1}^N \alpha_i > \alpha \sum_{j=1}^r \pdeg (f_j).
      \]
       Then $p$ divides the cardinality of $\{ \vb{a} \in \prod_{i=1}^N A_i \mid
      f_1 (\vb{a}) = \cdots =  f_r (\vb{a}) = 0 \}$.
    \end{thm}
     Actually, we shall prove the following stronger version.
      \begin{thm} \label{thm:rest2}
       Let $p$ be a prime, $\alpha \in \N$, and 
       let $F$ be a finite field with $q = p^{\alpha}$ elements.
       Let $f_1, \ldots, f_r \in F[x_1, \ldots, x_N]$, let
       $A$ be a subgroup of $\algop{F^N}{+}$ with
       $p^{M}$ elements. 
          We assume that
      \[
         M  > \alpha \sum_{j=1}^r \pdeg (f_j).
      \]
      Then $p$ divides the cardinality
      of $\{ \vb{a} \in A  \mid
      f_1 (\vb{a}) = \cdots =  f_r (\vb{a}) = 0 \}$.
      \end{thm}
      \begin{proof}
        We want to use Theorem~\ref{thm:warngroup}
        with $N := 1$,  $m := M$, $n := \alpha$,
        $\alpha_i = \beta_j = 1$ for $i \in \ul{M}$,
        $j \in \ul{\alpha}$; then $A \cong \Z_p^{M}$ and
        $B \cong \Z_p^{\alpha}$.
        We will now verify that with these settings, the
        inequality in the assumption of Theorem~\ref{thm:warngroup}
        is satisfied. 
        We  have
        $N (\sum_{i=1}^M (p^{\alpha_i} - 1)) =
        M (p - 1) >
        (\alpha \sum_{j=1}^r \pdeg (f_j))(p-1)$.
         By Theorem~\ref{thm:fp}, we have
        $\pdeg (f_j) \ge \fundeg (\ol{f_j})$ for each $j$,
        and hence by Lemma~\ref{lem:rest},
        $\pdeg (f_j) \ge \fundeg (\ol{f_j}|_A)$.
        Thus $(\alpha \sum_{j=1}^r \pdeg (f_j)) (p-1)
        \ge (\sum_{j=1}^r \fundeg(\ol{f_j}|_A)) \, \alpha \,(p-1) =
            (\sum_{j=1}^r \fundeg(\ol{f_j}|_A)) (\sum_{i=1}^{\alpha} (p - 1)) =
        (\sum_{j=1}^r \fundeg(\ol{f_j}|_A)) (\sum_{i=1}^{\alpha} (p^{\beta_i} - 1))$.
        Theorem~\ref{thm:warngroup} now yields
        that $p$ divides $|\{ \vb{a} \in A  \mid
        f_1 (\vb{a}) = \cdots =  f_r (\vb{a}) = 0 \}|$.
      \end{proof}
    Theorem~\ref{thm:rest1} is an immediate consequence of
    this result, since $A := \prod_{i=1}^N A_i$ is a subgroup
    of $\algop{F^N}{+}$ with $p^M$ elements, where
    $M=\sum_{i=1}^N \alpha_i$.
    In the case that the subsets $A_i$ are
    subgroups of $\algop{F}{+}$, then
    Theorem~\ref{thm:rest0} can then be derived from
    Theorem~\ref{thm:rest1} in the following way:
    Let $p$ be a prime and $\alpha \in \N$ be such that
    $q = p^{\alpha}$, and
    assume that for each $i \in \ul{N}$, $A_i$ is a subgroup of
    $\algop{F}{+}$ with $p^{\alpha_i}$ elements, and
    that
    $\sum_{i=1}^N (p^{\alpha_i} - 1) > (p^{\alpha} - 1) \sum_{j=1}^r \deg (f_j)$.
    In order to show that the assumptions of Theorem~\ref{thm:rest2}
    are fulfilled, we first
    estimate
    $\sum_{j = 1}^r \deg_p (f_j)$. Since $s_p (n) \le n$ for all $n \in \N$,
    we have
    $\sum_{j = 1}^r \deg_p (f_j) \le
    \sum_{j = 1}^r \deg   (f_j)$. By the assumption,
    $\sum_{j = 1}^r \deg   (f_j) < \sum_{i=1}^N ((p^{\alpha_i} - 1) / (p^{\alpha} -1)) = \sum_{i = 1}^N (q^{\alpha_i / \alpha} - 1) / (q - 1)$.
    Now we consider the function $e(x) = (q^x - 1) / (q - 1)$ in
    the real interval $[0,1]$. We have $e(0) = 0$ and $e(1) = 1$.
    Since $e$ is convex, we therefore have $e(x) \le x$ for all $x \in [0,1]$.
    Thus $\sum_{i = 1}^N (q^{\alpha_i / \alpha} - 1) / (q - 1) \le
    \sum_{i = 1}^N (\alpha_i / \alpha)$,
    and therefore $\alpha \sum_{j=1}^r \deg_p (f_j) < \sum_{i=1}^N \alpha_i = M$. 
    Applying Theorem~\ref{thm:rest1} we obtain that the set of common zeroes of
    the $f_j$'s has cardinality divisible by $p$.
 
     The bound in the assumptions of Theorem~\ref{thm:warngroup}
     can be improved if we know that the functions $f_i$ are
     not surjective. For this improvement, we first need an
     auxiliary result.
     \begin{lem} \label{lem:c0}
       Let $p$ be a prime, let $B$ be an abelian group of exponent $p$,
       and let $S$ be a finite subset of $B$ with $0 \in S$.
       Then there is a function $c_0^S : B \to \Z_p$ such that
       $c_0^S (0) = 1$, $c_0^S (s) = 0$ for $s \in S \setminus \{0\}$,
       and $\fundeg (c_0^S) \le |S| - 1$.
     \end{lem}
     \begin{proof}
       We proceed by induction on $|S|$. For $|S| = 1$,
       $c_0^S (x) := 1$ satisfies the required properties.
       If $|S| > 1$, we choose $s \in S \setminus \{0\}$.
       Since $B$ is of exponent $p$ and hence a vectorspace
       over $\Z_p$, there is a homomorphism $h : B \to \Z_p$
       with $h(s) = 1$. We choose $c_0^{S \setminus \{s\}}$ using the
       induction hypothesis and define
       $c_0^S (x) := (1 - h(x)) \cdot c_0^{S \setminus \{s\}} (x)$.
       Since $\fundeg (h (x)) = 1$, we obtain
       $\fundeg (c_0^{S}) \le 1 + (|S|- 2) = |S| - 1$.
     \end{proof}   
       
     \begin{thm}  \label{thm:restrange}
           Let $p$ be a prime, let $k \in \N$,
    let $\alpha_1, \ldots, \alpha_k \in \N_0$,
    let $A := \prod_{i=1}^k \Z_{p^{\alpha_i}}$,  let
    $B$ be an abelian group of exponent $p$, and let
    $\delta (A, \Z_p) = \sum_{i=1}^k (p^{\alpha_i} - 1)$.
    Let $f_1, \ldots, f_r : A \to B$.
    If
    \[
    \delta (A, \Z_p) > \sum_{i = 1}^r (\left|\range(f_i) \right| - 1) \fundeg (f_i),
    \]
    then $p$ divides $\left| V(f_1, \ldots, f_r) \right|$.
     \end{thm}
     \begin{proof}
       Let $S_i := \range (f_i)$.
       If there is an $i \in \ul{r}$ such that
       $0 \not\in \range (f_i)$, then $V(f_1, \ldots, f_r) = \emptyset$,
       and the claim is true. Hence we can assume
       $0 \in S_i$ for all $i \in \ul{r}$.
       Now consider the function
       \[
       f (x) := \prod_{i = 1}^r c_0^{S_i} (f_i (x)),
       \]
       where $x \in B$ and the functions $c_0^{S_i}$ are those constructed
       in Lemma~\ref{lem:c0}.
       We have $\fundeg (f) \le \sum_{i = 1}^r (|S_i| - 1) \fundeg (f_i)
       < \delta (A, \Z_p)$. Now Lemmas~\ref{lem:sum0} and~\ref{lem:degchar}  yield
       $\sum_{x \in A} f(x) = 0$. Since $f(x) = 1$ if $x \in V(f_1, \ldots, f_r)$
       and $f(x) = 0$ else, this implies that $p$ divides $\left|V(f_1, \ldots, f_r)
       \right|$.
     \end{proof}  
     
     \begin{cor} \label{cor:restrange}
       Let $p$ be a prime, let $\alpha \in \N$, and
       let $F$ be a finite field of characteristic $p$ with
       $q = p^{\alpha}$ elements, and let
       $f_1, \ldots, f_r \in F[x_1, \ldots, x_N]$.
       If
       \[
       N \alpha (p-1) > \sum_{i = 1}^r (\left|\range(\ol{f_i}) \right| - 1)
       \pdeg (f_i),
       \]
       then
       $p$ divides $\left|V(f_1, \ldots, f_r)\right|$.
    \end{cor}   
     \begin{proof}
       Let $A := F^N$. Then $A$ is isomorphic to $\Z_p^{\alpha N}$,
       and thus with the notation of  Theorem~\ref{thm:restrange}
       we have $\delta (A, \Z_p) = \alpha N (p-1)$.
       Now
       $\sum_{i = 1}^r (\left|\range( \ol{f_i}) \right| - 1) \fundeg (\ol{f_i})
           \le
           \sum_{i = 1}^r (\left|\range( \ol{f_i}) \right| - 1) \pdeg(\ol{f_i})
           < \delta (A, \Z_p)$, and
           hence Theorem~\ref{thm:restrange} yields the result.
     \end{proof}
     
\section{Warning's Second Theorem}
 Warning's Second Theorem states that if a system of $r$ polynomial 
 equations over a finite field with $q$ elements has a zero,  
 then it has at least $q^{n-d}$ zeros, where $d=\sum_{i=1}^r \deg(f_i)$.
 In the case $r=1$, this had been improved to $d = \pdeg (f_1)$ by
 \cite[Theorem~2]{MM:IOTC}.
 Using the functional degree in S.\ Asgarli's proof 
 of Warning's Second Theorem from 
 \cite{As:ANPO}, we obtain:
 \begin{thm} \label{thm:warningImprovement}
       Let $p$ be a prime, let $\alpha \in \N$,
       let $F$ be a finite field of characteristic $p$ with
        $q = p^{\alpha}$
    elements, let $r,N \in \N$, and
    let $f_1,\ldots,f_r \in F[x_1,\ldots,x_N]$.
    If $0 \in V(f_1,\ldots,f_r)$, 
      then $|V(f_1,\ldots,f_r)| \geq q^{N - \sum_{j=1}^r \pdeg(f_j)}$.
    \end{thm}
  \begin{proof}
    We first consider the case $\alpha = 1$. For this case,
    we follow the proof of S.\ Asgarli in \cite{As:ANPO}.
    Let $D := \sum_{i = 1}^r \fundeg (f_i)$.
    We first show
    \begin{equation} \label{eq:asg1}
       |V(f_1, \ldots, f_r)| \geq p^{N-D}.
    \end{equation}
    We proceed by induction on $N - D$.
    If $N \le D$, then
    $p^{N-D} \le 1$, and the existence
    of one solution is guaranteed by the assumption.
    For the induction step, we assume $N - D \ge 1$.
    Let $s := | V(f_1, \ldots, f_r) |$. 
    Then Theorem~\ref{thm:warngroupA} implies that $p$ divides $s$.
    As in \cite{As:ANPO}, we count
    \(
    X = \{ (x, H) \mid x \in V(f_1,\ldots, f_r) \setminus \{0\},\,
            H \text{ is a hyperplane in } \Z_p^N \text{ with } 0 \in H,\,
            x \in H \}
    \)
            in two ways.
            Since each $x \neq 0$ is contained in exactly
            $(p^{N-1} - 1)/(p-1)$ hyperplanes through the origin,
            we have
            $|X| = \frac{(s-1)(p^{N-1} - 1)}{p-1}$.
            Each hyperplane $H$ is the solution set of an equation $g = 0$
            with $\fundeg (g) = 1$.
            By the induction hypothesis, $|V(f_1,\ldots, f_r, g)| \ge p^{N-D-1}$.
            Hence, in each hyperplane $H$ with $0 \in H$,
            we find at least $p^{N-D-1}-1$ nonzero elements
            of $V(f_1, \ldots, f_r)$, and therefore
            $|X| \ge \frac{p^N - 1}{p-1} (p^{N-D-1} - 1)$.
            This implies
            $s-1 \ge \frac{p^N - 1}{p^{N-1} - 1} (p^{N-D-1} - 1) >
            p^{N-D} - p$.
            Since $p \mid s$, this implies $s \ge p^{N-D}$, completing
            the proof of~\eqref{eq:asg1}.

            Now let $\alpha$ be an arbitrary natural number.
            Taking a basis $(b_1, \ldots, b_{\alpha})$ of $F$ over
            $\Z_p$, we define
            $h_1, \ldots, h_{\alpha} : F \to \Z_p$ to be the group homomorphisms
            with $x = \sum_{i = 1}^{\alpha} h_i (x) b_i$ for all $x \in F$.
            Now we consider the system
            $h_1 \circ \ol{f_1} = \cdots = h_{\alpha} \circ \ol{f_1} =
            h_1 \circ \ol{f_2} = \cdots = h_{\alpha} \circ \ol{f_2} = \cdots
            = h_1 \circ \ol{f_r} = \cdots = h_{\alpha} \circ \ol{f_r} = 0$,
            where all $\alpha r$ functions map $\Z_p^{\alpha N}$ into
            $\Z_p$.
            From~\eqref{eq:asg1}, we know that this system
            has at least
            $p^E$ solutions, where
            $E = \alpha N - D'$ and
            $D' = \sum_{i=1}^{\alpha} \sum_{j=1}^r
            \fundeg (h_i \circ f_j)$.
            Since 
            $D' \le \sum_{i=1}^{\alpha} \sum_{j=1}^r  \fundeg (h_i) \fundeg (\ol{f_j})
                \le \sum_{i=1}^{\alpha} \sum_{j=1}^r  \fundeg (\ol{f_j})
            = \alpha \sum_{j=1}^r \fundeg (\ol{f_j}) \le
              \alpha \sum_{j=1}^r \pdeg (f_j)$,
              the result follows.
   \end{proof}
            
    \section*{Acknowledgements}
    The authors thank S.\ Fioravanti,  S.\ Kreinecker,
    P.\ Mayr,
    L.\ M\'{e}rai, C.\ Raab, G.~Regensburger, C.\ Szab\'o and
    A.\ Winterhof for discussions
    on the topics of this paper. 
    \bibliography{chev74}
    \end{document}